\documentclass{article}
\usepackage{color}
\usepackage{amsmath}
\usepackage{graphicx}
\usepackage{footnote}
\usepackage{amsfonts,color}
\usepackage{amsmath,amssymb}
\usepackage{mathtools}
\usepackage{epsfig}
\usepackage{amssymb} 
\usepackage{mathtools}
\usepackage{amsthm,wasysym}
\usepackage[english]{babel}
\usepackage{cancel}
\usepackage{colonequals}
\makeatletter
\usepackage{float}
\usepackage{wrapfig}
\usepackage{nicefrac,enumitem}
\usepackage{agn_tikz}
\usepackage{hyperref}
\usepackage{fixltx2e}
\restylefloat{figure}

\makeatletter
\let\@fnsymbol\@arabic
\makeatother
\usepackage{caption}
\usepackage{float}

\usepackage{bbm}
\newcommand{\id}{{\boldsymbol{\mathbbm{1}}}}

\newcommand{\tr}{{\rm tr}}

\newcommand{\norm}[1]{\|#1\|}

\def\dd{\displaystyle}

\newtheorem{theorem}{Theorem}[section]

\newtheorem{proposition}[theorem]{Proposition}
\newtheorem{corollary}[theorem]{Corollary}

\newtheorem{definition}[theorem]{Definition}

\newcommand{\todo}[1][]{\textcolor{red}{TODO\ifx&#1&\else: #1\fi}}

\def\dd{\displaystyle}

\usepackage{multicol}
\usepackage[font=small]{caption}
\newcommand{\citet}[2][]{\citeauthor{#2} \cite[#1]{#2}}
\RequirePackage{amsmath}	
\RequirePackage{mathtools}	
\newcommand{\col}{\colon}

\newcommand{\R}{\mathbb{R}}

\DeclareMathOperator{\GL}{GL}

\DeclareMathOperator{\SO}{SO}
\DeclareMathOperator{\OO}{O}

\DeclareMathOperator{\Sym}{Sym}

\newcommand{\GLp}{\GL^{\!+}}

\newcommand{\On}{\OO(3)}

\newcommand{\Symn}{\Sym(3)}

\DeclareMathOperator{\diag}{diag}

\DeclareMathOperator{\Cof}{Cof}

\newcommand{\pdd}[3][]{\frac{\partial\ifx&#1&\else^{#1}\fi #2}{\partial #3}}

\DeclareMathOperator{\@macros@div}{div}
\renewcommand{\div}{\@macros@div}

\newcommand{\innerproduct}[1]{\langle #1 \rangle}

\newcommand{\iprod}{\innerproduct}

\providecommand{\availableaturl}[2][]{%
	available at \url{#2}%
}

\usepackage[babel,german=quotes]{csquotes} 

\makeatletter%
\@ifpackageloaded{hyperref}{}{}%
\makeatother%

\usepackage{color}
\usepackage{amsmath}

\usepackage{footnote}
\usepackage{amsfonts,color}
\usepackage{amsmath,amssymb}
\usepackage{amssymb} 
\usepackage{mathtools}
\usepackage{amsthm,wasysym}
\usepackage{framed}
\usepackage{cancel}
\usepackage{graphicx}
\usepackage{caption}
\usepackage{subcaption}
\usepackage{rotating}

\usepackage[english]{babel}
\usepackage[utf8]{inputenc}
\makeatletter
\usepackage{float}
\usepackage{wrapfig}
\restylefloat{figure}

\makeatletter
\let\@fnsymbol\@arabic
\makeatother

\usepackage{bbm}

\def\dd{\displaystyle}

\def\barr{\begin{array}}
	\usepackage{lscape}

	\def\earr{\end{array}}
\def\bec#1{\begin{equation}\label{#1}}
\def\becn{\begin{equation*}}
\def\endec{\end{equation}}
\def\endecn{\end{equation*}}
\def\dd{\displaystyle}
\def\bfm#1{\mbox{\boldmat}}

\renewcommand{\dd}{\displaystyle}

\usepackage{enumitem}

\usepackage{pifont}

\tikzset{
state/.style={
	rectangle,
	rounded corners,
	draw=black, very thick,
	minimum height=2em,
	inner sep=6pt,
	text centered,
}
}
\newcommand{\imin}[1][\lambda]{m{\ifx&#1&\else(#1)\fi}}
\newcommand{\imax}[1][\lambda]{M{\ifx&#1&\else(#1)\fi}}
\newcommand{\iset}[1][\lambda]{J{\ifx&#1&\else(#1)\fi}}
\allowdisplaybreaks[1]

\usepackage{titletoc}
\begin{document}
\title{\vspace*{-2cm} The Biot stress - right stretch  relation for  the compressible Neo-Hooke-Ciarlet-Geymonat model and Rivlin's cube problem
}
\author{ 
	Ionel-Dumitrel Ghiba\,\thanks{Ionel-Dumitrel Ghiba,   Alexandru Ioan Cuza University of Ia\c si, Department of Mathematics,  Blvd. Carol I, no. 11, 700506 Ia\c si,
		Romania;  Octav Mayer Institute of Mathematics of the
		Romanian Academy, Ia\c si Branch,  700505 Ia\c si, Romania, email: dumitrel.ghiba@uaic.ro}\ , \quad Franz Gmeineder\,\thanks{Franz Gmeineder,  Department of Mathematics and Statistics, University of Konstanz, Universit\"atsstrasse 10, 78457 Konstanz, Germany,
		email: franz.gmeineder@uni-konstanz.de}\  , \quad  Sebastian Holthausen\,\thanks{Sebastian Holthausen,  Lehrstuhl f\"{u}r Nichtlineare Analysis und Modellierung, Fakult\"{a}t f\"{u}r Mathematik, Universit\"{a}t Duisburg-Essen, Thea-Leymann Str. 9, 45127 Essen, Germany, email: sebastian.holthausen@uni-due.de}\ ,\\ Robert J. Martin\,\thanks{Robert J. Martin,  Lehrstuhl f\"{u}r Nichtlineare Analysis und Modellierung, Fakult\"{a}t f\"{u}r Mathematik, Universit\"{a}t Duisburg-Essen, Thea-Leymann Str. 9, 45127 Essen, Germany, email: robert.martin@uni-due.de} \quad
	and\quad Patrizio Neff\,\thanks{Patrizio Neff,  Head of Lehrstuhl f\"{u}r Nichtlineare Analysis und Modellierung, Fakult\"{a}t f\"{u}r Mathematik, Universit\"{a}t Duisburg-Essen,  Thea-Leymann Str. 9, 45127 Essen, Germany, email: patrizio.neff@uni-due.de}}

\maketitle

\begin{abstract}

The aim of the  paper is to recall the importance of the study of invertibility and monotonicity of stress-strain relations for investigating the non-uniqueness and bifurcation of homogeneous solutions of the equilibrium problem of a hyperelastic cube subjected to equiaxial tensile forces.  In other words, we reconsider a remarkable possibility in this nonlinear scenario: Does  symmetric loading lead only to symmetric deformations or also to asymmetric deformations? If so, what can we say about monotonicity for these homogeneous solutions, a property which is less restrictive than the energetic stability criteria of homogeneous solutions for  Rivlin's cube problem. 
For the Neo-Hooke type materials we establish what  properties the volumetric function $h$ depending on $\det F$ must have to ensure the existence of a unique radial solution (i.e. the cube must continue to remain a cube) for any magnitude of radial stress acting on the cube.  The function $h$ proposed by  Ciarlet and Geymonat satisfies  these conditions. However, discontinuous equilibrium trajectories may occur, characterized by abruptly appearing non-symmetric deformations with increasing load, and  a cube can instantaneously become a parallelepiped.   Up to the load value for which the bifurcation in the radial solution is realized local monotonicity holds true. However, after exceeding this value, monotonicity no longer occurs on homogeneous deformations which, in turn,  preserve the cube shape. 

\end{abstract}

\setcounter{tocdepth}{3}\vspace*{-0.4cm}
\tableofcontents
\section{Introduction}

The theory of nonlinear elasticity is undoubtedly applicable in numerous contexts. However, depending on the specific phenomena we aim to analyze, different types of elastic energies come into play. Various materials exhibit different behaviors in terms of elasticity. In hyperelasticity, as considered here, stress is determined by the elastic energy density, making the selection of an energy function a crucial constitutive decision. The assumptions regarding the stress-strain relationship are referred to as constitutive requirements.

Therefore, one main task in hyperelasticity is to find an energy (or at least a family of energies) describing the behaviour of all,  or at least  a large class of materials. This question was raised by Clifford A.\ Truesdell (1919-2000) in ``Das ungel\"oste Hauptproblem der endlichen Elastizit\"atstheorie. Zeit. Angew. Math.  Mech. 36 (3-4):  97-103, 1956". At present, however, there is no  mathematical model in classical nonlinear elasticity
which is capable of describing the correct physical or mechanical behaviour for every elastic material, especially for large strains and for which the existence of the minimizer of the corresponding variational problem or the Euler-Lagrange equations is ensured.

For different type of materials or for various behaviours which we wish to capture in the modelling process, we must choose an appropriate energy. In this contribution, we reconsider the classical compressible polyconvex Neo-Hooke-type energies (Hadamard materials) \cite{tarantino2008homogeneous} 
\begin{align}\label{ine}
W_{\rm NH} (F)=\frac{\mu }{2}\|F\|^2+ h(\det F),
\end{align}
where $h$ is a convex function\footnote{In order to have a stress free configuration, the function $h$ must satisfy
	$
	\frac{3}{2}\mu +h^\prime(1)=0.$}. Here, we pay special attention to the Ciarlet-Geymonat-type energy \footnote{Note that the original Ciarlet-Geymonat energy reads
	\begin{align}
	W_{\rm CG} (F)=a \,\|F\|^2+b\,\|{\rm Cof}\, F\|^2+c\,(\det F)^2-d\, \log(\det \, F),
	\end{align}
	where $a, b, c, d$ are positive constants. This energy is polyconvex and agrees with the Saint-Venant energy  quadratic in  the Green-St Venant strain tensor $ E=\frac{1}{2}(C-\id)$. For the original Ciarlet-Geymonat model it follows  that the associated minimization problem has at least
	one solution by Ball’s theorem \cite{Ball77}. Due to weak coercivity, the same is not known for \eqref{hcg}.} \cite{ciarlet1982lois,ciarlet1982quelques} for compressible materials, i.e., when the function $h$ is of the form
\begin{align}\label{hcg}
h_{\rm CG}(x)=-\mu \,\log x+\frac{\lambda}{4}(x^2-2\, \log x-1), \qquad \lambda>0,
\end{align}
with given positive constitutive parameters $\mu $ and $\lambda$. In the following, we shall refer to this model as the Neo-Hooke-Ciarlet-Geymonat model.

This paper revisits the issue of monotonicity and invertibility of a fundamental relationship: the Biot stress-right stretch  tensor relation. We use recent analytical findings in this area as a first step towards addressing the non-symmetric bifurcation in the Rivlin cube problem associated with the Neo-Hooke-Ciarlet-Geymonat \cite{ciarlet1982lois,ciarlet1982quelques} energy model. In order to understand these phenomena, we incorporate new insights pertaining to these constitutive equations. Studying the deformations of a uniformly loaded cube \cite{rivlin1974stability} is an important  challenge within finite elasticity, revealing intriguing behaviours even for simple energy functions. Despite the straightforward mathematical setup of this equilibrium problem, its resolution can pose difficulties due to its inherent nonlinearity. Understanding the stability of solutions and the local monotonicity will add another layer of complexity.

The equilibrium problem of a cube under equitriaxial tensions was initially explored by Rivlin \cite{rivlin1974stability} and then by Rivlin and Beatty \cite{rivlin2003dead}, by Reese and Wriggers \cite{reese1997material}, by Mihai,Woolley and Goriely \cite{mihai2019likely} and by Tarantino \cite{tarantino2008homogeneous}, who all exposed multiple solutions, particularly in the realm of incompressible Neo-Hookean materials and compressible Neo-Hookean materials, respectively.  Surprisingly, these solutions may lack symmetry, deviating from the (perhaps) expected behavior even under symmetric external loads. Rivlin further found that only one solution maintains full symmetry with the loading conditions, but becomes unstable under higher tensile loads.
Ball and Schaeffer \cite{ball1983bifurcation} have further explored into the case of incompressible Mooney-Rivlin materials, discovering the possibility of secondary bifurcations, a phenomenon absent in the Neo-Hookean scenario. They applied techniques from singularity theory to study the local behaviour around bifurcation points.

The paper by Tarantino \cite{tarantino2008homogeneous}  aims to analyse both symmetric and asymmetric equilibrium configurations of bodies composed of general compressible isotropic materials. Special focus is put on exploring non-unique equilibrium states and relevant bifurcation phenomena. Building upon previous contributions by Ball  \cite{ball1984differentiability} and Chen \cite{chen1987stability,chen1995stability}, in  \cite{tarantino2008homogeneous}  a stability analysis is proposed to evaluate the stability of various homogeneous  equilibrium branches, see also \cite{soldatos2006stability}. In the present paper, we  rediscover some results obtained by Tarantino \cite{tarantino2008homogeneous}, but we  {rely} on some pertinent analytical explanations, and present them in relation to our new results concerning the invertibility and monotonicity in nonlinear elasticity \cite{MartinVossGhibaNeff}.

After a presentation of the problem and the general framework, in this article we establish results on the invertibility and monotonicity of stress-strain relations.  These findings will subsequently be used for the Biot stress tensor-left stretch tensor relation and in the study of the Rivlin cube problem. For Neo-Hooke type materials we establish what  properties the volumetric function $h$ depending on $\det F$ in \eqref{ine} must have to ensure the existence of a unique radial solution (i.e., the cube must continue to remain a cube) for any magnitude of radial stress acting on the cube. 

In particular, we prove that the function $h\equiv h_{CG}$ in \eqref{ine}  defining the Ciarlet-Geymonat energies has these properties. For the Neo-Hooke-Ciarlet-Geymonat model, after identifying the radial solutions, we identify the existence of non-radial solutions (i.e., the cube turns into a parallelepiped)  for the extension case. These solutions do not exist for the case of compression or if the magnitude of the forces does not exceed a certain critical value $\alpha^\flat$. Moreover, for radial and non-radial solutions the problem of monotonicity is studied using the results from \cite{MartinVossGhibaNeff}. Specifically, we prove that radial solutions ensure local monotonicity up to the critical value $\alpha^*\geq\alpha^\flat$ of the forces acting on the faces of the cube. This is where bifurcation occurs, i.e., the solution is no longer locally unique and beyond this value monotonicity no longer occurs in radial solutions. In this regard, we prove that the critical value $\alpha^*$ corresponds to the critical values of the stretch for which the invertibility in terms of the principal stress-principal stretch relation \cite{MartinVossGhibaNeff} is lost in radial solutions.

Starting from a value of the magnitude  $\alpha^\flat$ of the forces (less than the critical value that produces the bifurcation) there are two other types of non-radial solutions (other types are obtained by permutations of them). We show that these types of non-radial solutions cannot all have different eigenvalues but certainly at least two are equal. Returning to the type of non-radial solutions, both appear in a {\bf discontinuous manner} for a  value $\alpha^\flat$ of the magnitude of the forces and then depend continuously on the intensity of the forces.  One class of non-radial solutions continuously moves towards and through the bifurcation branch, while the other moves away from the bifurcation point. 
Numerical tests have shown that, while the first class does not ensure local monotonicity, the second one does enjoy monotonicity. It is for this reason that the latter solution  meets the physical expectations in a better way.  

\section{Statement of the problem}\setcounter{equation}{0}

Let $\Omega\subset{\R^3}$  be a bounded domain with Lipschitz boundary
$\partial\Omega$. A mapping $\varphi\col\Omega\to\R$ describes the deformation of the domain $\Omega$. The domain $\Omega$ is called the initial (undeformed) configuration, while its image $\Omega_c:=\varphi(\Omega)$ is called the actual (deformed) configuration. Each of these configurations could be considered as  a reference configuration, depending on the practical problem we solve or model.  Since we do not allow for self-intersection of the material,  there exists the inverse mapping $\varphi^{-1}\col\Omega_c\to\Omega$ from the deformed configuration to its initial configuration. Therefore, imposing the preservation of orientation,  the  \emph{deformation gradient} defined by
$
F\colonequals{\rm D}\varphi\in \mathbb{R}^{3\times 3}
$
satisfies $F\in\GLp(3)$, i.e., $J=\det F>0$.   For further notation, the reader is referred to Appendix \ref{notations}.

From a geometric or analytic point of view, this would suffice for a complete description of the deformation. However, in elasticity theory we assume that the domain $\Omega$ is filled by an elastic body. Thus, the aim is to take into account the physical response of the body, meaning the constitutive relation between stress (internal forces) and strain (amount of deformation). 
	In the context of nonlinear hyperelasticity, where generalized convexity properties have an especially long and rich history \cite{Ball77,knowles1976failure,knowles1978failure}, the material behaviour of an elastic solid is described by a potential energy function
	$
	W\col\GLp(3)\to\R\,,\  F\mapsto W(F)
	$
	defined on the group $\GLp(3)$ of invertible matrices with positive determinants. 
		In hyperelasticity, the tensor which describes the force of the deformed material per original area  (stress) is  the first Piola-Kirchhoff stress tensor, denoted here by $S_1(F)$, and    the stress-strain relation is described by the energy density potential $W:\GLp(3)\to \mathbb{R}$, through $S_1={\rm D}_F[W(F)]$.  We assume that  the material is homogeneous. 
	The elastic energy potential $W\col\GLp(3)\to\R$ is also  assumed to be \emph{objective} (or \emph{frame-indifferent}) as well as \emph{isotropic}, i.e., it to satisfy
	$
	W(Q_1F\,Q_2) = W(F)
	\ \text{for all }  \;F\in\GLp(3)
	\ \text{and all }\; Q_1,Q_2\in\SO(3)\,.
	$
	 Hence $W(F)=\widehat{W}(U)$, where $U$ is  the right
	stretch tensor, i.e.\ the unique element of ${\rm Sym}^{++}(3)$ for which $U^2=C:=F^TF$; here and throughout, ${\rm Sym}^{++}(3)$ denotes the positive definite, symmetric tensors.

	In the absence of body forces, the general boundary value problem  is to find the solution $\varphi$ of the equilibrium equation
	\begin{align}\label{eq}
	{\rm Div} \,S_1({\rm D} \varphi)=0\qquad  \text{in} \qquad \Omega\subset \R^{3},
	\end{align}
	where $S_1({\rm D} \varphi)$ is the first Piola-Kirchhoff stress tensor, subject to the boundary conditions
	\begin{align}\label{bc}
	S_1.\, N=\widehat{s}_1.
	\end{align}
	Here, $N$ is the unit normal at the boundary $\partial \Omega$ and the vector $\widehat{s}_1$ is given. 
	
In this article we study the invertibility of the stress-stretch relation, the monotonicity and the bifurcation problem for a dead loading problem. In \textbf{Rivlin's cube problem} \cite[page 15]{Marsden83}(see also \cite{ball1983bifurcation,wan1983symmetry,rivlin2003dead,tarantino2008homogeneous}) the unit cube is subjected to equal and opposite normal dead loads on all faces.\footnote{``Rivlin's solutions have been known for nearly half a century. Nevertheless, we have yet to find an experiment that demonstrates these solutions" \cite{chen1996stability}.} \textit{Dead loading} is a simple example of a traction boundary condition where  $\widehat{s}_1({x})$ is a constant vector at each point of the boundary $\partial \Omega$. Thus, the boundary conditions on the face of the unit coordinate normal $N_i$ are
\begin{align}\label{bcr}
S_1.\, N_i=\alpha\, N_i,\qquad\qquad i=1,2,3,
\end{align}
with $\alpha$ indicating the amount of load.

A minimizer  $\varphi\in C^2(\Omega)$ of the total energy functional  given by 
\begin{align}\label{totale1}
I(\varphi)=\int_{\Omega} W({\rm D}  \varphi)\, dV-\int_{\partial \Omega} \langle S_1.\,N, \varphi\rangle\, dA=\int_{\Omega} W({\rm D}  \varphi)\, dV-\int_{\partial \Omega} \langle \widehat{s}_1, \varphi\rangle\, dA
\end{align} 
is a solution  of the  boundary value problem given by \eqref{eq} and \eqref{bc}.

For Rivlin's cube problem, the body is subjected to a uniform load on the boundary $S_1.\, N_i=\alpha\, N_i,\, i=1,2,3$. An application of the divergence theorem yields that  the total energy functional is given by
\begin{align}\label{totale}
I(\varphi)=\int_{\Omega} W({\rm D}  \varphi) dV-\int_{\partial \Omega} \alpha\,\langle N, \varphi\rangle dA=\int_{\Omega}[W({\rm D} \varphi) -\alpha\, {\rm div}\varphi ]\, dV=\int_{\Omega}[W({\rm D} \varphi) -\alpha\, \tr({\rm D} \varphi)]\, dV.
\end{align} 
A deformation $\varphi$ is  homogeneous if the deformation gradient ${\rm D} \varphi$ is constant in $\Omega$. For a homogeneous deformation the equilibrium equations are immediately satisfied, while the boundary conditions give rise to a  system of nonlinear algebraic equations.

We recall that \cite[page 144]{Truesdell65}
\begin{align}\label{difs}
T_{\rm Biot}=&R^T S_1,
\end{align}
where 	$T_{\rm Biot}=R^TS_1(F)={\rm D}_U[\widehat{W}(U)]$
is the symmetric Biot stress tensor and $R$ is the orthogonal matrix of the polar decomposition $F=R\,U=V R$, see \cite{neff2013grioli,borisov2019optimality,fischle2017grioli}; here, $V=\sqrt{F\,F^T}$ is the left stretch tensor. It is known that 
\begin{center}
	\fbox{
		$T_{\rm Biot}$ is symmetric and represents ``the principal forces acting in the reference system".}
\end{center}

Since no rotations are present in the cube problem ($R=\id$),  we are able to rephrase Rivlin's cube problem \cite{Ogden83} as the algebraic nonlinear system
\begin{align}\label{Tbiot11}
T_{\rm Biot}(U)N_i=\alpha N_i\quad \iff\quad (T_{\rm Biot}(U)-\alpha\cdot\id) N_i=0,\qquad i=1,2,3,
\end{align}
where  $N_i$ are the linearly  independent normals to the faces. Here, we adopt the convention that only homogeneous solutions $U$ are  considered.  Therefore, \eqref{Tbiot11} is equivalent to
\begin{align}\label{Tbiot1}
T_{\rm Biot}(U)=\alpha\cdot \id.
\end{align}

For the Neo-Hooke model, Rivlin has shown that the problem \eqref{Tbiot1} admits several homogeneous  solutions\footnote{The radial solution is an abbreviation for the equitriaxial stretching $\lambda_1=\lambda_2=\lambda_3=\beta^+$.}, see Figure \ref{bifu}, and for a certain load parameter $\alpha$ the (always existing) homogeneous radial solution $U=\beta^+\id$, $\beta^+>0$, becomes unstable. Thus, an initially homogeneous and perfectly isotropic material would not behave as we expect intuitively from an isotropic material.
\begin{figure}[h!]
	\centering
	\includegraphics[scale=0.45]{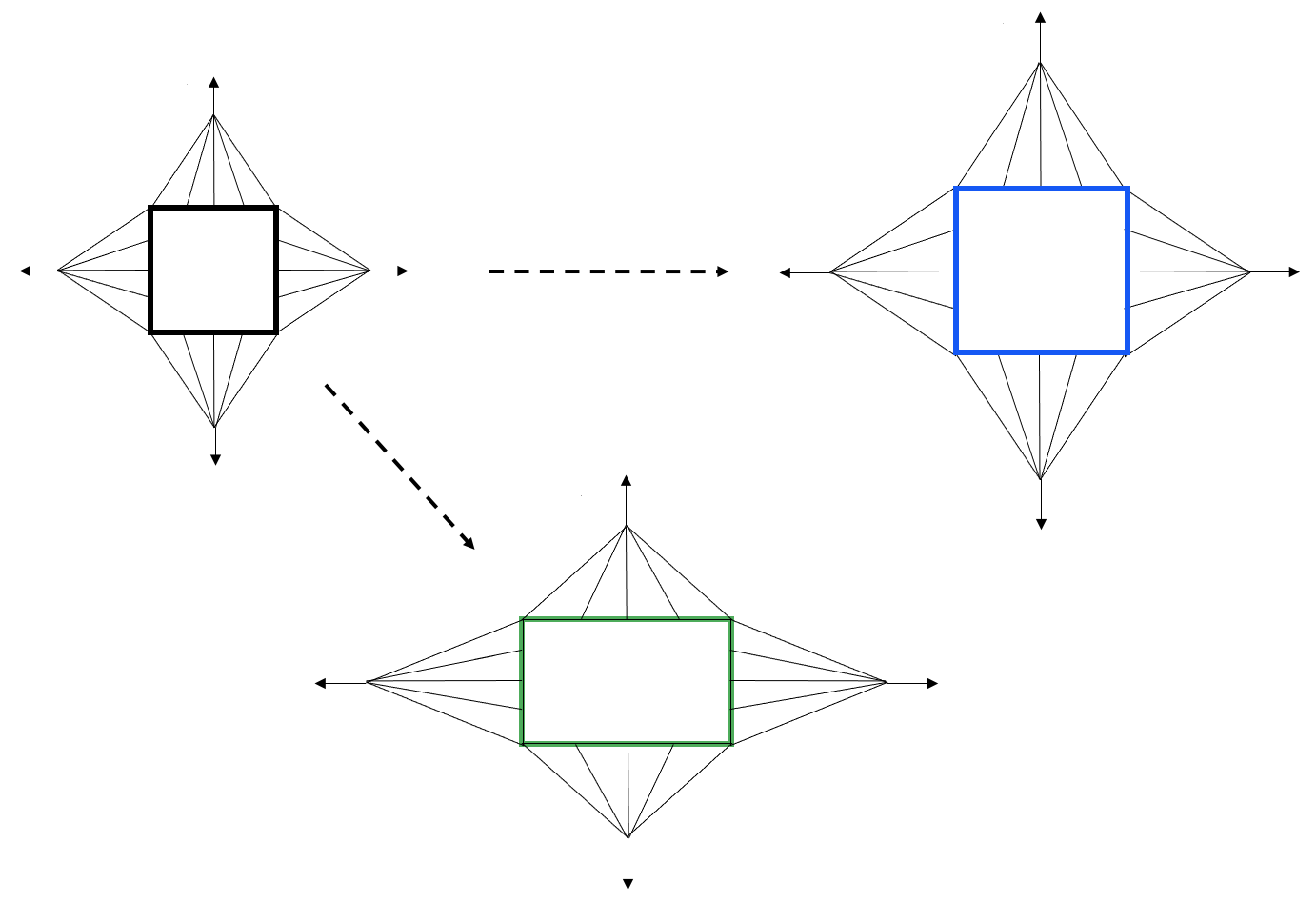}
	\put(-257,145){\footnotesize{$1$ }}
	\put(-265,137){\footnotesize{$1$ }}
	\put(0,145){\footnotesize{$f$ }}
	\put(-58,145){\footnotesize{$\lambda_1$ }}
	\put(-65,135){\footnotesize{$\lambda_1$ }}
	\put(-80,52){\footnotesize{$f$ }}
	\put(-150,52){\footnotesize{$\lambda_1$ }}
	\put(-165,43){\footnotesize{$\lambda_2$ }}
	\put(-218,154){\footnotesize{expected ideal symmetric }}
	\put(-218,140){\footnotesize{(radial) elastic solution}}
	\put(-280,90){\footnotesize{non-symmetric,}}
	\put(-280,80){\footnotesize{non radial solution}}
	\put(-280,70){\footnotesize{physically admissible}}
	\put(-280,60){\footnotesize{elastic solution (?)}}
	\caption{\footnotesize{Bifurcation for the compressible Rivlin's cube dead load problem:  equal and opposite normal dead loads on all faces.}}
	\label{bifu}
\end{figure}
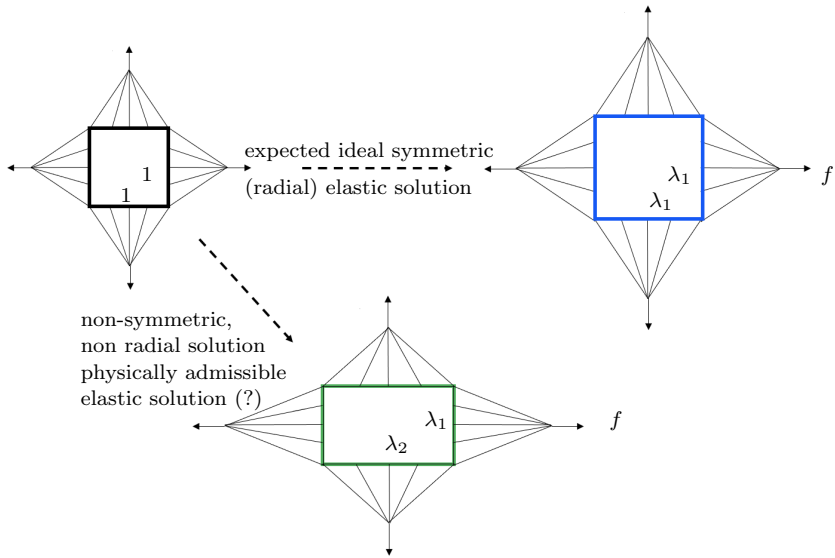%
Whether this can be really observed in experiments remains an open question. Discussing the related notion of Kearsley's instability, Batra et al.\ \cite[pages 710-711]{batra2005treloar} ``must conclude, rather prosaically, that Treloar's observation of two different stretches for equal loads is nothing else but another example of a notorious quality of rubber, namely the difficulty of quantitative reproducibility of rubber data and the unreliability of exact numbers obtained from rubber experiments", but provide new experimental data that indeed supports Kearsley's claims of instabilities in rubber sheets (cf.~\cite{gent2005elastic}).

For Neo-Hookean materials, the problem \eqref{Tbiot1} admits several non-symmetric homogeneous solutions \break $U\in {\rm Sym}^{++}(3).$
Based on \eqref{difs} we are able to investigate several equivalent statements in terms of different stress tensors. Since in Rivlin's cube problem there is no cause for a non-symmetric response, it is questionable if there should be any non-symmetric response\footnote{Of course, rubber is not at all incompressible under high pressure; rather, for moderate pressure, rubber ``tries" to respond in a way which preserves volume due to a comparatively low shear modulus compared to the bulk modulus}.

Ideally, we aim at the radial solution to be  locally unique among all other solutions.  In particular, we  insist on the logical rule that there is ``no effect without a corresponding cause". Since in Rivlin's cube problem there is no cause for a non-symmetric response, there should be no admissible non-symmetric response\footnote{``Experimentally" observed non-symmetric bifurcations seem to be inevitably accompanied by  permanent deformations \cite{batra2005treloar}.}. If this is or is not the case, depends on the chosen constitutive relation.

If there is   a radial  solution $U=\beta^+\,\id$ of Rivlin's cube problem then {\it invertibility}  and {\it monotonicity} of $T_{\rm Biot}:{\rm Sym}^{++}(3)\to {\rm Sym}(3)$  suffice to exclude  symmetric bifurcations altogether, as it will be seen in the rest of the paper.

\section{Constitutive requirements in nonlinear elasticity}\setcounter{equation}{0}
	\subsection{Invertibility }\label{ips}

	We consider the  general isotropic constitutive equation
	\begin{equation}
	\Sigma=\Sigma(U)\qquad \qquad \Sigma:{\rm Sym}^{++}(3)\to \Sym(3)
	\end{equation}
	where $\Sigma$ is some symmetric stress tensor and $U\in{\rm Sym}^{++}(3)$ is the stretch tensor. We are then interested in the following two important questions:
	\begin{itemize}
		\item[i)] (surjectivity) given any symmetric  tensor $\Sigma\in{\rm Sym}(3)$, does  there exist a positive definite tensor \break $U\in{\rm Sym}^{++}(3)$ such that $\Sigma=\Sigma(U)$?
		\item[ii)] (injectivity) for a given symmetric  tensor $\Sigma\in{\rm Sym}(3)$, does there exist at most  one  $U\in{\rm Sym}^{++}(3)$ such that $\Sigma=\Sigma(U)$?
	\end{itemize}
	
	It is clear that when  an idealised model is proposed (hence, no elasto-plastic response is expected) the  first requirement seems mandatory. The second requirement is the first step in order to exclude  bifurcation  \cite{rivlin1974stability} for a dead loading problem.

	We have observed \cite{MartinVossGhibaNeff} a possible way to study the invertibility of the map $U\mapsto T_{\rm Biot}(U)$. 	Since, in general, it is not easy to work with tensors (matrices) in three dimensions, we consider the singular values (the principal stretches) $\lambda_1$, $\lambda_2$, $\lambda_3$ of $F$, i.e., the positive eigenvalues  of $U$.
If $\Sigma_f\col\Symn\to\Symn$ is an	 isotropic tensor function satisfying \begin{align}\Sigma_f(Q^T\cdot\diag(\lambda_1,\lambda_2,\lambda_3)\cdot Q)=Q^T\cdot\Sigma_f(\diag(\lambda_1,\lambda_2,\lambda_3))\cdot Q\end{align} then 
\begin{align}
 \Sigma_f(U):=\Sigma_f(\underbrace{Q^T\cdot\diag(\lambda_1,\lambda_2,\lambda_3)\cdot Q}_{S\,\in\,\Sym(3)}) = \underbrace{Q^T\cdot\diag(f(\lambda_1,\lambda_2,\lambda_3))\cdot Q}_{\Sigma_f(S)\,\in\,\Sym(3)} \quad\forall\,Q\in\On\label{eq:introductionMatrixFunction}
\end{align}
with a vector-function $f=(f_1,f_2,f_3)\col\R^3\to\R^3$ which fulfills 
\[
f_i(\lambda_{\pi(1)},\lambda_{\pi(2)},\lambda_{\pi(3)}) = f_{\pi(i)}(\lambda_1,\lambda_2,\lambda_3)
\]
for any permutation $\pi\col\{1,2,3\}\to\{1,2,3\}$. Here, $\Symn$ denotes the space of symmetric $3\times 3$ matrices, $\On$ is the orthogonal group and $\diag(\lambda_1,\lambda_2,\lambda_3)$ is the diagonal matrix with diagonal entries $\lambda_1,\lambda_2,\lambda_3\,.$ The permutation symmetry implies and it is implied by isotropy.

Indeed, in many situations the stress-strain relations are characterized by the relation between their corresponding principal values, i.e., by the relations between the principal stretches $\lambda_1, \lambda_2, \lambda_3 $ and the  principal forces (principal Biot-stresses)
  \begin{align}\label{BT}T_i= \dd\frac{\partial g(\lambda_1,\lambda_2,\lambda_3)}{\partial \lambda_i},\ i=1,2,3.\end{align} 
 In \eqref{BT}, $g\colon\mathbb{R}_+^3\to \mathbb{R}$ is the unique permutation symmetric function  of the singular values of $U$ (principal stretches) such that $W(F)=\widehat{W}(U)=g(\lambda_1,\lambda_2,\lambda_3)$, where $\mathbb{R}_+^3=(0,\infty)\times(0,\infty)\times (0,\infty)$, and 
 \begin{align}
 \widehat{T}:=\begin{pmatrix}
 T_1,&T_2,&T_3
 \end{pmatrix}^T.
 \end{align}
The functions $f$ and $\Sigma_f$ related by eq.~\eqref{eq:introductionMatrixFunction} share a number of properties related to invertibility and monotonicity.
\begin{theorem}[\cite{MartinVossGhibaNeff}]\label{theorem:invertibility}
	Let $f\col\R^3_+\to\R^3$ be symmetric.
	\begin{itemize}
		\item[i)] The function $\Sigma_f\col{\rm Sym}^{++}(3)\to\Symn$ is injective if and only if $f$ is injective.
		\item[ii)] The function $\Sigma_f\col{\rm Sym}^{++}(3)\to\Symn$ is surjective if and only if $f$ is surjective.
	\end{itemize}
	In particular, $\Sigma_f$ is invertible if and only if $f$ is invertible.
\end{theorem}


\noindent In particular, using the result by  Katriel \cite{katriel2016mountain} (see also \cite{galewski2014conditions})  proving the global homeomorphism theorem of Hadamard, we obtain a sufficient criterion for the global invertibility of an isotropic tensor function.
\begin{proposition}\label{invprop}
	Assume that $\widetilde{\Sigma}_{\widetilde{f}}:{\rm Sym}(3)\to{\rm Sym}(3)$ is an isotropic $C^1$-function defined by the vector-function ${\widetilde{f}}:\mathbb{R}^3\to \mathbb{R}^3$ such that \begin{enumerate}
		\item ${\rm D}{\widetilde{f}}\,(x_1,x_2,x_3)$ is invertible for any $(x_1,x_2,x_3)\in \mathbb{R}^3$;
		\item $\|{\widetilde{f}}\,(x_1,x_2,x_3)\|_{\mathbb{R}^3}\to \infty$ as $\|(x_1,x_2,x_3)\|_{\mathbb{R}^3}\to \infty$.
		
	\end{enumerate}
	Then $U\mapsto \widetilde{\Sigma}_{\widetilde{f}}(U)$ is a global diffeomorphism from ${\rm Sym}(3)$ to ${\rm Sym}(3)$.
\end{proposition}

Let us remark that Proposition~\ref{invprop} is not directly applicable to $\Sigma_f:{\rm Sym}^{++}(3)\to{\rm Sym}(3)$. However, we have the following corollary to Katriel's result: 
\begin{corollary}\label{corinvprop}
	Assume that $\Sigma_f:{\rm Sym}^{++}(3)\to{\rm Sym}(3)$ is an isotropic $C^1$-function such that\begin{enumerate}
		\item ${\rm D}(f\circ \exp)\,(x_1,x_2,x_3)$ is invertible\footnote{Here, $\exp\,(x_1,x_2,x_3)=(\exp x_1,\exp x_2,\exp x_3)$ for any $(x_1,x_2,x_3)\in \mathbb{R}^3$.} for any $(x_1,x_2,x_3)\in \mathbb{R}^3$;
		\item $\|(f\circ \exp)\,(x_1,x_2,x_3)\|_{\mathbb{R}^3}\to \infty$ as $\|(x_1,x_2,x_3)\|_{\mathbb{R}^3}\to \infty$.
	\end{enumerate}
	Then $U\mapsto \Sigma_f(U)$ is a global diffeomorphism from ${\rm Sym}^{++}(3)$ to ${\rm Sym}(3)$.
\end{corollary}
\begin{proof}
	Let us consider $\widetilde{\Sigma}_{\widetilde{f}}:{\rm Sym}(3)\to{\rm Sym}(3)$ defined by
	\begin{align}
	\widetilde{\Sigma}_{\widetilde{f}}(S):=(\Sigma_f\circ \exp)(S) \quad \forall\ S\in \Sym(3)\qquad \Longleftrightarrow\qquad  \widetilde{\Sigma}_{\widetilde{f}}(\log U):=\Sigma_f(U) \quad \forall\ U\in \Sym^{++}(3),
	\end{align}
	where $\log U=\sum\limits_{i=1}^3 \log \lambda_i\,  N_i\otimes N_i,$ with $N_i$  the eigenvectors of $U$ and $\lambda_i$  the eigenvalues of
	$U$, is the Hencky strain tensor \cite{Neff_Osterbrink_Martin_hencky13,NeffGhibaLankeit}. The  function $\widetilde{f}$ defining $\widetilde{\Sigma}_{\widetilde{f}}$ is $\widetilde{f}:=f\circ \exp:\mathbb{R}^3\to \mathbb{R}^3$.
	Now,  Katriel's result applied to $\widetilde{\Sigma}_{\widetilde{f}}$ shows that if ${\rm D}\widetilde{f}\,(x_1,x_2,x_3)$ is invertible for any $(x_1,x_2,x_3)\in \mathbb{R}^3$ and $\|\widetilde{f}\,(x_1,x_2,x_3)\|_{\mathbb{R}^3}\to \infty$ as $\|(x_1,x_2,x_3)\|_{\mathbb{R}^3}\to \infty$, then $U\mapsto \widetilde{\Sigma}_{\widetilde{f}}(U)$   is a global diffeomorphism from ${\rm Sym}(3)$ to ${\rm Sym}(3)$. Then, since the matrix logarithm function $\log:{\rm Sym}^{++}(3)\to {\rm Sym}(3)$ is a global diffeomorphism, $U\mapsto \Sigma_f(U)=(\widetilde{\Sigma}_{\widetilde{f}}\circ \log)( U)$ must be a global diffeomorphism as well.
\end{proof}

\begin{corollary}\label{corinvprop1}
	Assume that $\Sigma_f:{\rm Sym}^{++}(3)\to{\rm Sym}(3)$ is an isotropic $C^1$-function such that\begin{enumerate}
		\item ${\rm D}f\,(\lambda_1,\lambda_2,\lambda_3)$  is invertible for any $(\lambda_1,\lambda_2,\lambda_3)\in \mathbb{R}^3_+$;
		\item $\|f(\lambda_1,\lambda_2,\lambda_3)\|_{\mathbb{R}^3}\to \infty$ as  $\|(\log\lambda_1,\log\lambda_2,\log\lambda_3)\|_{\mathbb{R}^3}\to \infty$.
	\end{enumerate}
	Then $U\mapsto \Sigma_f(U)$ is a global diffeomorphism from ${\rm Sym}^{++}(3)$ to ${\rm Sym}(3)$.
\end{corollary}
\begin{proof}

	First, by using the chain rule and the invertibility of $D\exp\ $, we observe that the assumption that ${\rm D}f\,(\lambda_1,\lambda_2,\lambda_3)$  is invertible for any $(\lambda_1,\lambda_2,\lambda_3)\in \mathbb{R}^3_+$ is equivalent to the invertibility of ${\rm D}(f\circ \exp)\,(x_1,x_2,x_3)$  for any $(x_1,x_2,x_3)\in \mathbb{R}^3$.

Consider now the condition $\|(x_1,x_2,x_3)\|_{\mathbb{R}^3}\to \infty$. We will prove that under assumption 2.\ in the Corollary, it follows that $\|(f\circ \exp)\,(x_1,x_2,x_3)\|_{\mathbb{R}^3}\to \infty$. Indeed, let $(\lambda_1,\lambda_2,\lambda_3)\in \mathbb{R}^3_+$ be such that $x_i=\log \lambda_i$ for $i=1,2,3$. Then, $\|(\log \lambda_1,\log \lambda_2,\log \lambda_3)\|_{\mathbb{R}^3}\to \infty$. From 2., this implies that 
	 $\|f(\lambda_1,\lambda_2,\lambda_3)\|=\|(f\circ \exp)(x_1,x_2,x_3)\|\to \infty$.
	 
	 Therefore, the requirements of Corollary \ref{corinvprop} are satisfied, and this implies that  $U\mapsto \Sigma_f(U)$ is a global diffeomorphism from ${\rm Sym}^{++}(3)$ to ${\rm Sym}(3)$.
\end{proof}

\subsection{Hilbert-monotonicity }\label{ips2}

For our purposes, we now recall some related notions of monotonicity.

\begin{definition}\cite{NeffMartin14}
A tensor function $\Sigma_f\col{\rm Sym}^{++}(3)\to\Symn\,$ is called  \textbf{strictly  Hilbert-monotone} if
\begin{align}
\iprod{\Sigma_f(U)-\Sigma_f(\overline U),\,U-\overline U}_{\R^{3\times 3}}>0\qquad\forall\,U\neq\overline U\in{\rm Sym}^{++}(3)\,.\label{eq:introductionMatrixMonotonicity}
\end{align}
We refer to this inequality as \textbf{strict Hilbert-space matrix-monotonicity} of the tensor function $\Sigma_f$. A tensor function $\Sigma_f\col{\rm Sym}^{++}(3)\to\Symn\,$ is called  \textbf{Hilbert-monotone} if
\begin{align}
\iprod{\Sigma_f(U)-\Sigma_f(\overline U),\,U-\overline U}_{\R^{3\times 3}}\geq 0\qquad\forall\,U,\,\overline U\in{\rm Sym}^{++}(3)\,.\label{eq:introductionMatrixMonotonicity1}
\end{align}
\end{definition}
\begin{definition}\cite{NeffMartin14}
 A vector function  $f\col\R^3_+\to\R^3$ is \textbf{strictly vector monotone} if 
\begin{align}
\iprod{f(\lambda)-f(\overline\lambda),\,\lambda-\overline\lambda}_{\R^3}>0\qquad\forall\lambda\neq\overline\lambda\in\R^3_+,
\end{align}
and it is is   \textbf{vector monotone} if 
\begin{align}
\iprod{f(\lambda)-f(\overline\lambda),\,\lambda-\overline\lambda}_{\R^3}\geq 0\qquad\forall\lambda,\lambda\in\R^3_+,
\end{align}
\end{definition} 
\begin{definition}
	A continuously differentiable tensor function $\Sigma_f\col{\rm Sym}^{++}(3)\to\Symn\,$ is called  \textbf{strongly Hilbert-monotone} if
	\[
		\iprod{{\rm D}\Sigma_f.H(U),H} > 0 \qquad\text{for all }\;{U\in\rm Sym}^{++}(3),\;H\in{\rm Sym}(3)
		\,.
	\]
\end{definition}

\begin{definition}\cite{NeffMartin14}
	A continuously differentiable vector function  $f\col\R^3_+\to\R^3$ is called \textbf{strongly vector monotone} if
	\[
		\iprod{{\rm D}f(\lambda).h,h} > 0 \qquad\text{for all }\;\lambda\in\R^3_+,\;h\in\R^3
		\,.
	\]
\end{definition}


\noindent Note that ${\rm D}f\,(\lambda_1,\lambda_2,\lambda_3)$ in itself might not be symmetric. However, for $T_i= \dd\frac{\partial g(\lambda_1,\lambda_2,\lambda_3)}{\partial \lambda_i},\ i=1,2,3 $,
\begin{align}
{\rm D}\widehat{T}(\lambda_1,\lambda_2, \lambda_3)={\rm D}^2 g(\lambda_1,\lambda_2, \lambda_3):=\left(\frac{\partial^2 g}{\partial \lambda_i\partial \lambda_j}\right)_{i,j=1,2,3}\in \Sym(3).
\end{align}

\noindent In a forthcoming paper \cite{MartinVossGhibaNeff}, we discuss the following result, thereby expanding on Ogden's work \cite[last page in the Appendix]{Ogden83}, following Hill's seminal contributions  \cite{hill1968constitutivea,hill1968constitutiveb,hill1970constitutive}: 
\begin{theorem}
	\label{theorem:mainResult}
	A sufficiently regular symmetric function $f\col\R^3_+\to\R^3$ is (strictly/strongly) vector-monotone if and only if $\Sigma_f$ is (strictly/strongly) matrix-monotone.
\end{theorem}
\noindent Hence, the following  holds true for hyperelasticity, assuming sufficient regularity:
\begin{align*}
U\mapsto T_{\rm Biot}(U) \ \ \text{Hilbert-monotone}\quad &\Longleftrightarrow\quad (\lambda_1,\lambda_2,\lambda_3)\mapsto \widehat{T}(\lambda_1,\lambda_2, \lambda_3) \ \ \text{vector monotone}\notag\\&\Longleftrightarrow\quad {\rm D}\widehat{T}(\lambda_1,\lambda_2, \lambda_3)\in {\rm Sym}^{+}(3)\quad\forall(\lambda_1,\lambda_2, \lambda_3)\in\mathbb{R}_+^3,
\\[.7em]
U\mapsto T_{\rm Biot}(U) \ \ \text{strictly Hilbert-monotone}\quad &\Longleftrightarrow\quad (\lambda_1,\lambda_2,\lambda_3)\mapsto \widehat{T}(\lambda_1,\lambda_2, \lambda_3) \ \ \text{strictly vector monotone,}\notag
\\[.7em]
U\mapsto T_{\rm Biot}(U) \ \ \text{strongly Hilbert-monotone}\quad &\Longleftrightarrow\quad (\lambda_1,\lambda_2,\lambda_3)\mapsto \widehat{T}(\lambda_1,\lambda_2, \lambda_3) \ \ \text{strongly vector monotone}\notag\\&\Longleftrightarrow\quad {\rm D}\widehat{T}(\lambda_1,\lambda_2, \lambda_3)\in {\rm Sym}^{++}(3)\quad\forall(\lambda_1,\lambda_2, \lambda_3)\in\mathbb{R}_+^3.
\end{align*}
Note that the monotonicity conditions and the invertibility condition are global conditions. Conversely, the conditions $ \det {\rm D}f\,(\lambda_1,\lambda_2,\lambda_3)\neq 0$ -- which is equivalent to $f$ being a local diffeomorphism -- as well as ${\rm D}\widehat{T}\in {\rm Sym}^{++}(3)$ are only local conditions.


\subsection{Energetic stability}\label{ips3}


In the following, we employ the stability criterion
\begin{align}\label{ssc1}
\langle {\rm D}^2_{F}W(F)\, .H,H\rangle\geq 0 \qquad \text{for all} \qquad H\in \mathbb{R}^{3\times 3}
\end{align}
for the hyperelastic energy potential $W$, which ensures material stability under so-called soft loads \cite{chen1995stability}.
In terms of the singular values, the condition \eqref{ssc1} holds at $F\in\GLp(3)$ if and only if \cite{chen1987stability}
\begin{align}\label{staine1}
\frac{\frac{\partial g}{\partial \lambda_i}-\epsilon_{ij} \,\frac{\partial g}{\partial \lambda_j}}{\lambda_i-\epsilon_{ij} \,\lambda_j}\Big|_{\lambda_i=\lambda_i^*}\geq 0 \qquad\text{holds for all}\; i,j=1,2,3,  \ \ i\neq j \ \ \text{(no sum)}
\end{align}
and the Hessian matrix of $g$, i.e., ${\rm D}^2 g=\left(\frac{\partial^2 g}{\partial \lambda_i\partial \lambda_j}\right)\Big|_{\lambda_i=\lambda_i^*}
$
is positive semi-definite, where \begin{align}\epsilon_{ij}=\begin{cases}
\ \ \,1& \text{if}\ \ \{i,j\}=\{1,2\} \ \ \text{or}\ \ \{2,3\}\ \ \text{or}\ \ \{3,1\},\\
-1& \text{otherwise}
\end{cases}\end{align}
and $\lambda_i^*$ are the singular values of $F$.
If two singular values $\lambda^*_i$ and $\lambda^*_j$, $i\neq j$, are equal, the inequalities in \eqref{staine1} are interpreted in terms of their limits $\lambda^*_i\to \lambda^*_j$;
for instance, in the points $(\lambda^*_1,\lambda^*_1,\lambda^*_3)=(\lambda^*,\lambda^*,\lambda^*_3)$ with $\lambda_3^*\neq\lambda^*$, the energetic stability criterion \eqref{ssc1} is satisfied if and only if
\begin{align}\label{snr}
\left(\frac{\partial^2 g}{\partial \lambda_1^2}-\frac{\partial^2 g}{\partial \lambda_1\partial \lambda_2}\right)\Big|_{\lambda_1=\lambda_2=\lambda^*,\lambda_3=\lambda^*_3}&\geq 0,\qquad 
\left(\frac{\frac{\partial g}{\partial \lambda_2}-\frac{\partial g}{\partial \lambda_3}}{\lambda_2-\lambda_3}\right)\Big|_{\lambda_1=\lambda_2=\lambda^*,\lambda_3=\lambda^*_3}\geq 0,\\ \left(\frac{\frac{\partial g}{\partial \lambda_2}+\frac{\partial g}{\partial \lambda_1}}{2\lambda_1}\right)\Big|_{\lambda_1=\lambda_2=\lambda^*,\lambda_3=\lambda^*_3}&\geq 0,\qquad 
\left(\frac{\frac{\partial g}{\partial \lambda_2}+\frac{\partial g}{\partial \lambda_3}}{\lambda_2+\lambda_3}\right)\Big|_{\lambda_1=\lambda_2=\lambda^*,\lambda_3=\lambda^*_3}\geq 0,\notag
\end{align}
and the Hessian matrix ${\rm D}^2 g=\left(\frac{\partial^2 g}{\partial \lambda_i\partial \lambda_j}\right)\Big|_{\lambda_1=\lambda_2=\lambda^*,\lambda_3=\lambda^*_3}
$ is positive semi-definite.

We also remark that the positive semi-definiteness of the Hessian matrix of $g$ is equivalent to the positive semi-definiteness of $
{\rm D}\widehat{T}(\lambda_1,\lambda_2, \lambda_3)={\rm D}^2 g(\lambda_1,\lambda_2, \lambda_3).
$ Since the stability implies the positive semi-definiteness of $
{\rm D}^2 g(\lambda_1,\lambda_2, \lambda_3)={\rm D}\widehat{T}(\lambda_1,\lambda_2, \lambda_3)
$, the stability implies the monotonicity of $
{\rm D}\widehat{T}(\lambda_1,\lambda_2, \lambda_3)
$.

\section{Invertibility and monotonicity of the Biot stress-stretch relation for the compressible Neo-Hooke-Ciarlet-Geymonat energy}\setcounter{equation}{0}
In the following, we will reduce the Neo-Hooke-Ciarlet-Geymonat energy to its one-parameter version
\begin{align}\label{CGMe}
W^M_{\rm CG} (F)=\frac{1}{\mu}\,W_{\rm CG} (F)=\frac{1}{2}\|F\|^2+\left[-\,\log \det F+\left(\frac{M}{4}-\frac{1}{6}\right)((\det F)^2-2\, \log \det F-1)\right], 
\end{align}
\text{with} $M:=\frac{\lambda+\frac{2\,\mu}{3}}{\mu}>\frac{2}{3}$.
All the stresses considered in the following will be related to this one parameter energy. 
In terms of the singular values, $W^M_{\rm CG}$ admits the representation
\begin{align}
W^M_{\rm CG} (F)=g(\lambda_1,\lambda_2,\lambda_3)=
\frac{1}{2} \Big[&\frac{1}{6} (3 M-2) \left(\lambda_1^2 \lambda_2^2 \lambda_3^2-2 \log (\lambda_1 \lambda_2 \lambda_3)-1\right)-2 \log (\lambda_1 \lambda_2 \lambda_3)+\lambda_1^2+\lambda_2^2+\lambda_3^2\Big].
\end{align}
The corresponding first Piola-Kirchhoff stress tensor is given by
\begin{align}
S_1 &= F +\frac{1}{\mu}\, (h_{\rm CG}^M)^\prime(\det F)\cdot {\rm Cof} F= F +\left[ \left(\frac{M}{2}-\frac{1}{3}\right) \left( \det F-\frac{1}{\det F}\right)-\frac{1}{\det F}\right]\cdot {\rm Cof} F, 
\end{align}
where
\begin{align}
h_{\rm CG}^M:=\frac{1}{\mu }h_{\rm CG}=-\log x+\frac{\lambda}{4\,\mu }(x^2-2\, \log x-1).
\end{align}
The Biot stress tensor defined by $W^M_{\rm CG} (F)$ is
\begin{align}
T_{\rm Biot} (U)={\rm D}_U W^M_{\rm CG} (U)=R^T S_1&=U +\frac{1}{\mu } \, (h_{\rm CG}^M)^\prime(\det F)(\det U)\cdot \det U\cdot U^{-1}\notag\\&=U + \,\left[ \left(\frac{M}{2}-\frac{1}{3}\right) \left( \det U-\frac{1}{\det U}\right)-\frac{1}{\det U}\right]\cdot \det U\cdot U^{-1},
\end{align}
while the principal Biot stresses are given by
\begin{align}\label{pbs}
T_1&=\lambda_1-\frac{1}{\lambda_1}+\left(\frac{M}{2}-\frac{1}{3}\right) \left( \lambda_1 \lambda_2^2 \lambda_3^2-\frac{1}{\lambda_1}\right), \quad 
T_2=\lambda_2-\frac{1}{\lambda_2}+\left(\frac{M}{2}-\frac{1}{3}\right) \left( \lambda_1^2 \lambda_2 \lambda_3^2-\frac{1}{\lambda_2}\right),\\
T_3&=\lambda_3-\frac{1}{\lambda_3}+\left(\frac{M}{2}-\frac{1}{3}\right) \left( \lambda_1^2 \lambda_2^2 \lambda_3-\frac{1 }{\lambda_3}\right).\notag
\end{align}
We compute
\begin{align}
{\rm D}\widehat{T}=\left(
\begin{array}{ccc}
\frac{1}{\lambda_1 ^2}+1 +\left(\frac{M}{2}-\frac{1}{3}\right) \left(\frac{1}{\lambda_1 ^2}+\lambda_2^2 \lambda_3^2\right)& 2 \left(\frac{M}{2}-\frac{1}{3}\right) \lambda_1  \lambda_2 \lambda_3^2 & 2 \left(\frac{M}{2}-\frac{1}{3}\right) \lambda_1  \lambda_2^2 \lambda_3 \\
2 \left(\frac{M}{2}-\frac{1}{3}\right) \lambda_1  \lambda_2 \lambda_3^2 & \frac{1}{\lambda_2^2}+1+\left(\frac{M}{2}-\frac{1}{3}\right) \left(\lambda_1 ^2 \lambda_3^2+\frac{1}{\lambda_2^2}\right) & 2 \left(\frac{M}{2}-\frac{1}{3}\right) \lambda_1 ^2 \lambda_2 \lambda_3 \\
2 \left(\frac{M}{2}-\frac{1}{3}\right) \lambda_1  \lambda_2^2 \lambda_3 & 2 \left(\frac{M}{2}-\frac{1}{3}\right) \lambda_1 ^2 \lambda_2 \lambda_3 & \frac{1}{\lambda_3^2}+1+\left(\frac{M}{2}-\frac{1}{3}\right) \left(\lambda_1 ^2 \lambda_2^2+\frac{1}{\lambda_3^2}\right) \\
\end{array}
\right)
\end{align}
and   remark that
\begin{align}
\det {\rm D}\widehat{T}(1,1,1)=12\,M,
\end{align}
which is strictly positive for all $M>0$. However, we find that for all $M>\frac{2}{3}$ there exists $(\lambda_1,\lambda_2, \lambda_3)\in \mathbb{R}_+^3$ such that 
\begin{align}
\det {\rm D}\widehat{T}(\lambda_1,\lambda_2, \lambda_3)=0.
\end{align}
A quick numerical check reveals that the Biot stress-stretch relation is in general not invertible, see Figure \ref{3Dinv}. However, we may equally show this analytically. Indeed, for each material of the form \eqref{CGMe} with $M>\frac{2}{3}$ we have
\begin{align}
\det {\rm D}\widehat{T}(\lambda_1,\lambda_1, \lambda_1)=\frac{\left[(2-3 \,M) \lambda_1^6+6\, \lambda_1^2+4+3\, M\right]^2 \left[5\, (3 M-2)\, \lambda_1^6+6\, \lambda_1^2+4+3
	\, M\right]}{216\, \lambda_1^6},
\end{align}
and therefore, $U\mapsto T_{\rm Biot}(U)$ loses differentiable invertibility in $U=\lambda^*\, \id$, where  $\lambda_1$ is a solution of the equation (see Figure \ref{xm})
\begin{align}\label{losti}
(-3\, M+2)\, \lambda_1^6+6\, \lambda_1^2+4+3\, M=0. 
\end{align}

In Figure \ref{xm}, for fixed $M$, the solution is the intersection of the red line to the blue curve.  However, the analytical proof of the existence and uniqueness of the solution $\lambda^*$ of \eqref{losti} is also possible.

\restylefloat*{figure}

\begin{figure}[!h]
	\centering
	\begin{minipage}{.45\textwidth}
		\includegraphics[scale=0.5]{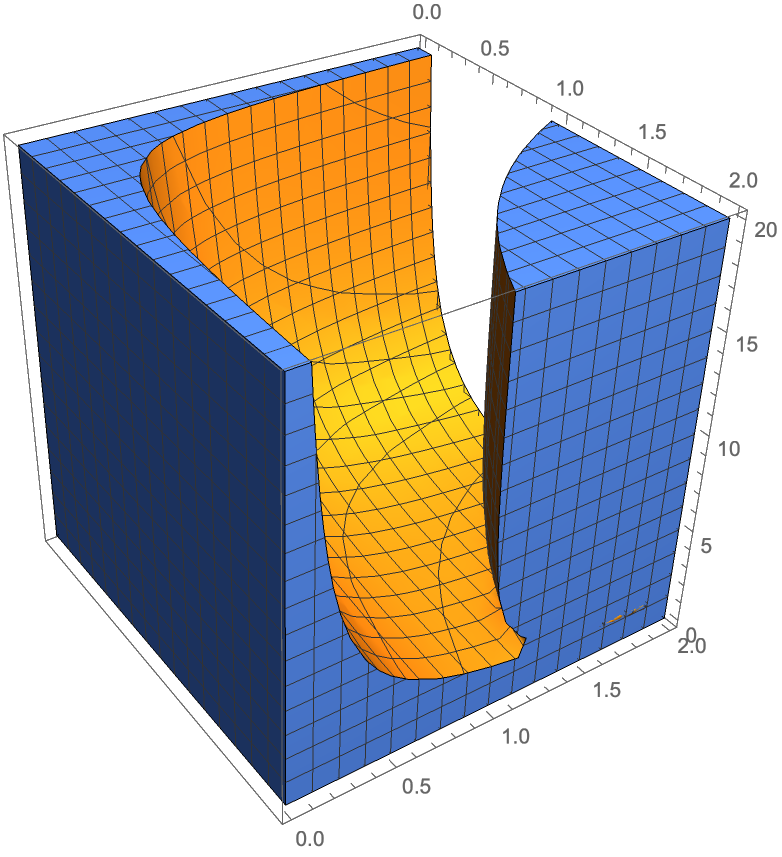}
		\caption{\footnotesize{For $M=1$, the region of those $(\lambda_1,\lambda_2,\lambda_3)$ for which $\det {\rm D}\widehat{T}(\lambda_1,\lambda_2, \lambda_3)\neq 0$.}}
		\label{3Dinv}
	\end{minipage}
	\qquad \quad 
	\begin{minipage}{.45\textwidth}
		\centering
		\includegraphics[scale=0.5]{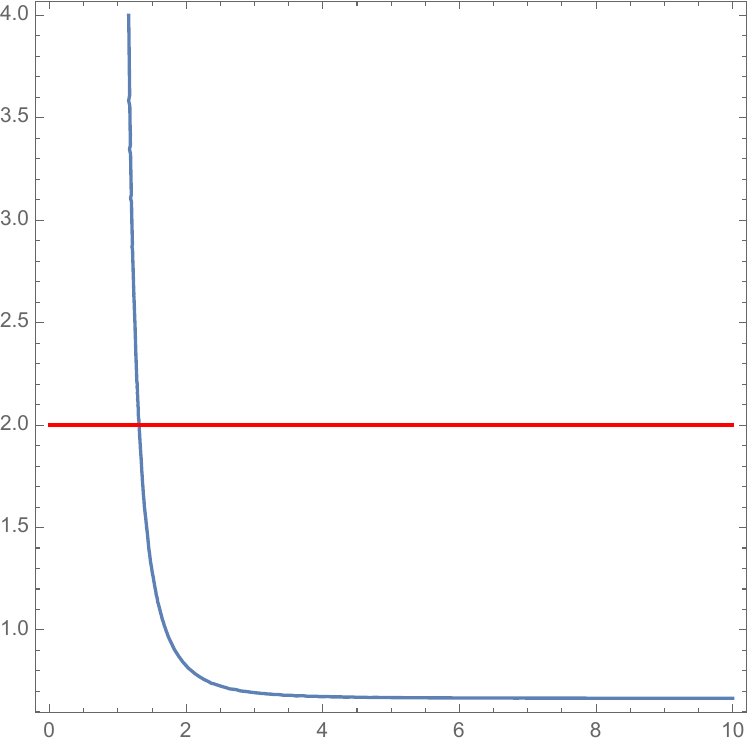}
		\put(0,0){\footnotesize $\lambda_1$}
		\put(-180,185){\footnotesize $M$}
		\put(-130,100){\footnotesize $(2-3 \,M)\, \lambda_1^6+6 \lambda_1^2+4+3 M>0$}
		\caption{\footnotesize{The plot of the pairs $(\lambda_1,M)$ satisfying \break $(2-3 \,M)\, \lambda_1^6+6 \lambda_1^2+4+3 M=0$. For fixed $M = 2$, the unique  solution is the intersection of the red line to the blue curve.}}
		\label{xm}
	\end{minipage}
\end{figure}%

\begin{figure}[!h]
	\centering
	\begin{minipage}{.45\textwidth}
		\includegraphics[scale=0.55]{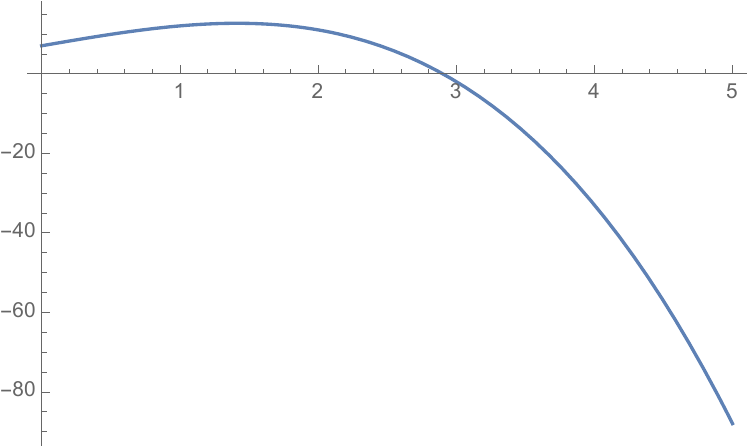}
		\caption{\footnotesize{For $M=1$, the plot of $s:(0,\infty)\to \mathbb{R},$ \\  $s(x)= ( - 3 M+2) x^3+6x+3 M +4$.}}
		\label{proofM1}
	\end{minipage}
	\qquad \quad 
	\begin{minipage}{.45\textwidth}
		\centering
		\includegraphics[scale=0.6]{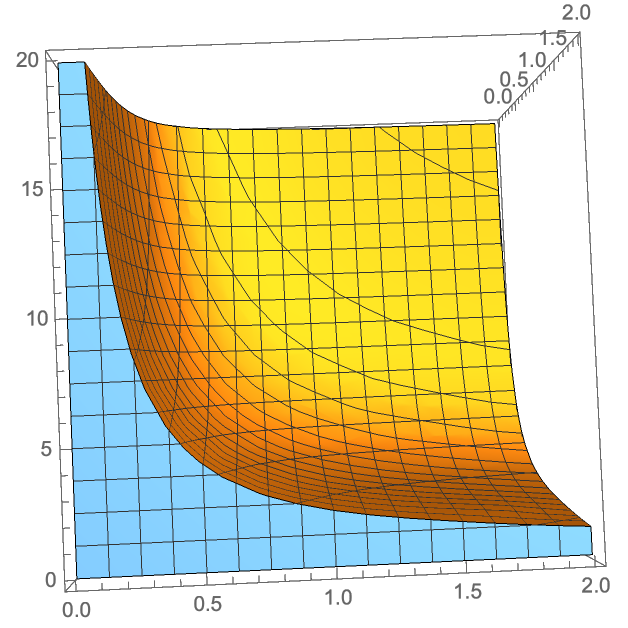}
		\caption{\footnotesize{For $M=1$, the region plot of those $(\lambda_1,\lambda_2,\lambda_3)$ for which the matrix  ${\rm D}\widehat{T}$ is positive definite.}}
		\label{3dm}
	\end{minipage}
	\qquad \quad 
\end{figure}

\begin{proposition}\label{notinv}
	For any $M>\frac{2}{3}$ the Biot stress-stretch relation $U\mapsto T_{\rm Biot}(U)$ for the Neo-Hooke-Ciarlet-Geymonat energy is in general not a diffeomorphism.
\end{proposition}
\begin{proof}
Let us consider the function
\begin{align}
s:(0,\infty)\to \mathbb{R}, \qquad s(x)= ( - 3 M+2) x^3+6x+3 M +4.
\end{align} 
Surely, we have 
\begin{align}
s(x^2)=
(-3\, M+2)\, x^6+6\, x^2+4+3\, M.
\end{align}
Therefore, $(-3\, M+2)\, x^6+6\, x^2+4+3\, M=0$ has a  positive solution $\lambda_1$ if and only if $s(x)=0$ has a positive solution. But the function $s$ is concave (see Figure~\ref{proofM1}), since
\begin{align}
s''(x)=6 (2-3 M) x<0 \quad \forall \  x>0 \quad \forall\  M>\frac{2}{3},
\end{align}
and it attains its maximum in the stationary point, i.e., in the solution  of the equation
\begin{align}s'(x)=0\quad \iff\quad  6 + (6 - 9 M) x^2=0\quad\Longleftrightarrow\quad x= \frac{\sqrt{2}}{\sqrt{3 M-2}}>0.
\end{align}
Note that 
\begin{align}
s(0) = 3\, M + 4 > 0, \qquad s\left(\frac{\sqrt{2}}{\sqrt{3 M-2}}\right)=3 M+\frac{4}{\sqrt{\frac{3 M}{2}-1}}+4>0 \quad \text{and}\quad \lim_{x\rightarrow \infty}s(x)=-\infty.
\end{align}
Thus, by the concavity, $s(x)$ remains positive at least until it reaches its maximum $x_0$ and, starting from $x_0$, $s(x)$ is strictly monotone decreasing. Since $s(x)$ is continuous and $\underset{x \to \infty}{\lim} \, s(x) \to - \infty$, there must be exactly one point $\widetilde x$, for which $s(\widetilde x) = 0$ by the intermediate value theorem and the strict monotonicity (starting from $x_0$), meaning that $\widetilde x$ is the unique solution to $s(x) = 0$.
\end{proof}

According to Theorem \ref{theorem:mainResult}, 
strong monotonicity  of the Biot stress-stretch relation for the Neo-Hooke-Ciarlet-Geymonat energy implies the positive semi-definiteness of the matrix ${\rm D}\widehat{T}$. Note however that, being not invertible and symmetric, the matrix  ${\rm D}\widehat{T}$ is also not positive definite everywhere.

Moreover, as visualized for $M=1$ via numerical simulation in Figure \ref{3dm},  the matrix  ${\rm D}\widehat{T}$ is  not positive semi-definite on $\mathbb{R}_+^3$ in general.
\begin{proposition}\label{thmon}
	For the compressible Neo-Hooke-Ciarlet-Geymonat materials,  the Biot stress-stretch relation  is in general not monotone.
\end{proposition}
\begin{proof}
	
	Note that $
	{\rm D}\widehat{T}(\lambda_1,\lambda_1, \lambda_1)$ is a symmetric $3\times 3$ matrix having the principal minors
	\begin{align}
	m^{\rm Biot}_1(\lambda_1,\lambda_1, \lambda_1)&:=\frac{\frac{1}{6} (3 M-2) \left(\lambda_1^6+1\right)+\lambda_1^2+1}{\lambda_1^2}>0 \quad \text{for}\quad M>\frac{2}{3},\notag\\
	m^{\rm Biot}_2(\lambda_1,\lambda_1, \lambda_1)&:=\frac{\left[3\, (3 M-2) \,\lambda_1^6+6 \lambda_1^2+4+3 M\right] \left[(2-3 M) \,\lambda_1^6+6\, \lambda_1^2+4+3\,M\right]}{36 \,\lambda_1^4},\\
	m^{\rm Biot}_3(\lambda_1,\lambda_1, \lambda_1)&:=\frac{\left[(2-3 M) \lambda_1^6+3 M+6 \lambda_1^2+4\right]^2 \left[5 (3 M-2) \lambda_1^6+3 M+6 \lambda_1^2+4\right]}{216 \lambda_1^6}.\notag
	\end{align}

	The curve of those $(\lambda_1,M)\in \mathbb{R}_+\times (\frac{2}{3},\infty)$ such that \eqref{losti}  is satisfied, divides the plane into two parts, one part at which $(2-3\, M)\, \lambda_1^6+6\, \lambda_1^2+4+3\, M>0$ (above the blue curve) and the part at which $(2-3\, M)\, \lambda_1^6+6\, \lambda_1^2+4+3\, M<0$ (below the blue curve). For each fixed $M>\frac{2}{3}$, see Figure \ref{xm}, for $\lambda_1<\lambda^*$, where $\lambda^*$ corresponds to the intersection point of the red curve with the blue curve,  the pairs $(\lambda_1,M)$ are on the left hand side of the blue, so $(2-3\, M)\, \lambda_1^6+6\, \lambda_1^2+4+3\, M<0$, while for $\lambda_1>\lambda^*$, the pairs $(\lambda_1,M)$ are on the right hand side of the blue, so $(2-3\, M)\, \lambda_1^6+6\, \lambda_1^2+4+3\, M>0$. 
	
	Therefore, we have 
	\begin{align}
	m^{\rm Biot}_1(\lambda_1,\lambda_1, \lambda_1)>0,\qquad 
	m^{\rm Biot}_2(\lambda_1,\lambda_1, \lambda_1)>0,\qquad 
	m^{\rm Biot}_3(\lambda_1,\lambda_1, \lambda_1)>0, \qquad \text{for}\qquad \lambda_1<\lambda^*,\notag\\
	m^{\rm Biot}_1(\lambda_1,\lambda_1, \lambda_1)>0,\qquad 
	m^{\rm Biot}_2(\lambda_1,\lambda_1, \lambda_1)=0,\qquad 
	m^{\rm Biot}_3(\lambda_1,\lambda_1, \lambda_1)=0, \qquad \text{for}\qquad \lambda_1=\lambda^*,\\
	m^{\rm Biot}_1(\lambda_1,\lambda_1, \lambda_1)>0,\qquad 
	m^{\rm Biot}_2(\lambda_1,\lambda_1, \lambda_1)<0,\qquad 
	m^{\rm Biot}_3(\lambda_1,\lambda_1, \lambda_1)>0, \qquad \text{for}\qquad \lambda_1>\lambda^*,\notag
	\end{align}
	and, according to the Sylvester criterion, the proof is complete.
\end{proof}

\section{Existence of the  radial solution for general Neo-Hooke models}\setcounter{equation}{0}

We  have shown that  the map $T_{\rm Biot}:{\rm Sym}^{++}(3)\to{\rm Sym}(3)$ is, in general, not a diffeomorhpism.
However, even if $T_{\rm Biot}$ is not surjective, in the construction of a  homogeneous solution of Rivlin's cube problem, this does not immediately imply that the equation \eqref{Tbiot1} does not have a solution. Moreover, since it is unclear yet whether $T_{\rm Biot}$ is injective, \eqref{Tbiot1}  could have more than one solution.  Equally, after the system 
\begin{align}\label{Tbiot12}
T_{\rm Biot}(U)=\alpha\cdot  \id\qquad \Longleftrightarrow \qquad T_i(\lambda_1,\lambda_2,\lambda_3)=\alpha, \qquad i=1,2,3
\end{align} is solved, one may ask whether $\widehat{T}(\lambda_1,\lambda_2,\lambda_3)$ or $ T_{\rm Biot}$ is locally strongly monotone in the solutions or if the homogeneous solutions are locally unique minimizers or energetically stable. Recall that the stability condition and global monotonicity were defined in Sections \ref{ips2} and \ref{ips3}, while  local strong monotonicity in ${U}\in \Sym^{++}(3)$ means that there exists $c_+>0$ such that for sufficiently small $\varepsilon>0$,
\begin{align}
\langle T_{\rm Biot}(\widetilde{U})- T_{\rm Biot}(U),\widetilde{U}-U\rangle>c_+\,\|\widetilde{U}-U\|^2 \qquad \forall \ \widetilde{U}\in \Sym^{++}(3) \qquad \text{such that} \quad \|\widetilde{U}-U\|<\varepsilon.
\end{align}
We also note that (local) strict monotonicity implies the (local) uniqueness of the solution of \eqref{Tbiot12}, since otherwise, assuming that $U_1$ and $U_2$ are two different solutions,
\begin{align}
\langle T_{\rm Biot}(U_1)- T_{\rm Biot}(U_2),U_1-U_2\rangle=0\,,
\end{align} 
which contradicts  the (local) strict monotonicity.

%

In this section we consider the general models for the classical  Neo-Hooke-type energies, i.e.,
\begin{align}\label{mate}
W_{\rm NH} (F)=\frac{\mu }{2}\langle {C},\id\rangle+ h(\det F)=\frac{\mu }{2}\|F\|^2+ h(\det F)=\frac{\mu }{2}\|U\|^2+ h(\det U).
\end{align}
The entire study is actually equivalent to the study of the one-parameter model described by the energy
\begin{align}
W_{\rm NH}^M (F):=\frac{1}{\mu }W_{\rm NH} (F)=\frac{1}{2}\langle {C},\id\rangle+ \frac{1}{\mu }\,h(\det F)=\frac{1}{2}\|U\|^2+ \frac{1}{\mu }\,h(\det U).
\end{align}
The corresponding first Piola-Kirchhoff stress tensor for this one parameter energy is given by
\begin{align}
S_1^{\rm NH} = F +\frac{1}{\mu } \, h^\prime(\det F)\cdot {\rm Cof} F,
\end{align}
and the Biot stress tensor is
\begin{align}
&T_{\rm Biot}^{\rm NH} (U)={\rm D}_UW_{\rm NH} (U)=R^T S_1=U +\frac{1}{\mu } \, h^\prime(\det U)\cdot \det U\cdot U^{-1}.
\end{align}

In order to have a stress free reference configuration, the function $h$ has to satisfy
$
\frac{3}{2}+\frac{1}{\mu }\,h^\prime(1)=0.
$
Since $\mu >0$, we have $h^\prime(1)<0$.

The first step in the study of the Rivlin cube problem is to check if a radial Biot stress tensor $T_{\rm Biot}^{\rm NH}=\alpha\,\id$ leads to a unique radial solution $U=\beta^+\, \id$ of the equation 
\begin{align}\label{eqNHcl}
T_{\rm Biot}^{\rm NH} (U)=\alpha\, \id.
\end{align}
\begin{proposition}\label{propNH}
For a hypperelastic material of the form \eqref{mate}, if the equation \eqref{eqNHcl} has a  unique radial solution $U=\beta^+\, \id$, $\beta^+>0$ for every $\alpha \in \mathbb{R}$, then the convex function $h$ satisfies
\begin{align}\label{propc1cl}
& \Big(\sqrt[3]{x}+\frac{1}{\mu }\, h^\prime(x)\sqrt[3]{x^2}\Big)^{\prime}\geq 0 \qquad \forall \ x>0,\\
&\lim_{x\rightarrow 0} \Big(\sqrt[3]{x}+\frac{1}{\mu } \, h^\prime(x)\sqrt[3]{x^2}\Big)=-\infty,\qquad
\lim_{x\rightarrow\infty} \Big(\sqrt[3]{x}+\frac{1}{\mu } \, h^\prime(x)\sqrt[3]{x^2}\Big)=\infty.\notag
\end{align}
If the convex function $h$ satisfies
\begin{align}\label{propc1clstrict}
& \Big(\sqrt[3]{x}+\frac{1}{\mu }\, h^\prime(x)\sqrt[3]{x^2}\Big)^{\prime}> 0 \qquad \forall \ x>0,\\
&\lim_{x\rightarrow 0} \Big(\sqrt[3]{x}+\frac{1}{\mu } \, h^\prime(x)\sqrt[3]{x^2}\Big)=-\infty,\qquad
\lim_{x\rightarrow\infty} \Big(\sqrt[3]{x}+\frac{1}{\mu } \, h^\prime(x)\sqrt[3]{x^2}\Big)=\infty,\notag
\end{align}
then  the equation \eqref{eqNHcl} has a  unique radial solution $U=\beta^+\, \id$, $\beta^+>0$ for every $\alpha \in \mathbb{R}$.
\end{proposition}
\begin{proof}

Equation \eqref{eqNHcl}, after  multiplication with $U$, reads
\begin{align}
U^2 +\frac{1}{\mu } \, h^\prime(\det U)\cdot \det U\cdot \id=\alpha\cdot U.
\end{align}
This system has a radial solution $U=\beta^+\cdot \id$ if $\beta^+$ is a solution to the equation
\begin{align}
\beta^++\frac{1}{\mu } \, h^\prime((\beta^+)^3)(\beta^+)^2=\alpha,
\end{align}
or with the substitution $x=(\beta^+)^3$, if there is a unique positive solution $x$ of  the equation
\begin{align}\label{xalpha}
\sqrt[3]{x}+\frac{1}{\mu }\, h^\prime(x)\sqrt[3]{x^2}=\alpha.
\end{align}

 There  exists at least one solution $x$ of the equation \eqref{xalpha} if and only if for each $\alpha\in \mathbb{R}$ the function $x\mapsto \sqrt[3]{x}+\frac{1}{\mu }\, h^\prime(x)\sqrt[3]{x^2}$ is not bounded on $(0,\infty)$. Otherwise, there exist values of $\alpha$, smaller or larger than the lower bound or upper bound, respectively, for which the function $x\mapsto \sqrt[3]{x}+\frac{1}{\mu }\, h^\prime(x)\sqrt[3]{x^2}$ never reach these values of $\alpha$. On the other hand, if the function is unbounded, then if the function $x\mapsto \sqrt[3]{x}+\frac{1}{\mu }\, h^\prime(x)\sqrt[3]{x^2}$ were not monotone, then the equation \eqref{xalpha} could have more than one solution for some $\alpha\in \mathbb{R}$.
In conclusion,  for a given $\alpha\in\R$, if the equation \eqref{eqNHcl} has a unique solution then the convex function $h$ has one of the following properties:
\begin{align}\label{propc1cl1}
& \Big(\sqrt[3]{x}+\frac{1}{\mu } \, h^\prime(x)\sqrt[3]{x^2}\Big)^{\prime}\geq 0 \qquad \forall x>0,\\
&\lim_{x\rightarrow 0} \Big(\sqrt[3]{x}+\frac{1}{\mu } \, h^\prime(x)\sqrt[3]{x^2}\Big)=-\infty,\quad
\lim_{x\rightarrow\infty} \Big(\sqrt[3]{x}+\frac{1}{\mu } \, h^\prime(x)\sqrt[3]{x^2}\Big)=\infty,\notag
\end{align}
or
\begin{align}\label{propc2cl2}
& \Big(\sqrt[3]{x}+\frac{1}{\mu } \, h^\prime(x)\sqrt[3]{x^2}\Big)^{\prime}\leq 0 \qquad \forall x>0,\\
&\lim_{x\rightarrow 0} \Big(\sqrt[3]{x}+\frac{1}{\mu } \, h^\prime(x)\sqrt[3]{x^2}\Big)=\infty,\quad
\lim_{x\rightarrow\infty} \Big(\sqrt[3]{x}+\frac{1}{\mu } h^\prime(x)\sqrt[3]{x^2}\Big)=-\infty.\notag
\end{align}

Since $h$ is convex, $h^\prime $ is monotone  increasing. Hence,
\begin{align}\label{semnconh}
h^\prime(x)>h^\prime(1) \ \forall\, x>1\qquad \text{and}\qquad h^\prime(x)<h^\prime(1)<0 \ \forall\, x<1.
\end{align}

The second set of conditions \eqref{propc2cl2} is therefore not admissible, since \eqref{propc2cl2}$_2$ implies that  $\lim_{x\rightarrow 0} \Big(\, h^\prime(x)\sqrt[3]{x^2}\Big)=\infty$, which is not possible (since  \eqref{semnconh} yields  $h^\prime(x)\sqrt[3]{x^2}<0$ for all $x<1$). Hence, it remains  that if  the system \eqref{eqNHcl} has  a unique radial solution, then $h$ has to satisfy the conditions \eqref{propc1clstrict}.

Finally, note that, the last part of the conclusions, the uniqueness of $\beta^+$ is implied by the strict monotonicity of  the mapping $x\mapsto \Big(\sqrt[3]{x}+\frac{1}{\mu } \, h^\prime(x)\sqrt[3]{x^2}\Big)$ and by the limit conditions.
\end{proof}

\section{Bifurcation in  Rivlin's cube problem for the compressible Neo-Hooke-Ciarlet-Geymonat model}\setcounter{equation}{0}
Let us  now consider the  Neo-Hooke-Ciarlet-Geymonat model, i.e., the Neo-Hooke model for which $\frac{1}{\mu } h(x)$ is given by  the function  \begin{align}
h_{\rm CG}^M(x)=-\log x+\left(\frac{M}{4\,}-\frac{1}{6}\right)(x^2-2\, \log x-1).
\end{align}

For the Neo-Hooke-Ciarlet-Geymonat model, Rivlin's cube problem amounts to finding the solutions of the nonlinear algebraic system
\begin{align}\label{pbs1}
T_1\equiv \lambda_1-\frac{1}{\lambda_1}+\left(\frac{M}{2}-\frac{1}{3}\right) \left( \lambda_1 \lambda_2^2 \lambda_3^2-\frac{1}{\lambda_1}\right)&=\alpha,\notag\\
T_2\equiv\lambda_2-\frac{1}{\lambda_2}+\left(\frac{M}{2}-\frac{1}{3}\right) \left( \lambda_1^2 \lambda_2 \lambda_3^2-\frac{1}{\lambda_2}\right)&=\alpha,\\
T_3\equiv\lambda_3-\frac{1}{\lambda_3}+\left(\frac{M}{2}-\frac{1}{3}\right) \left( \lambda_1^2 \lambda_2^2 \lambda_3-\frac{1 }{\lambda_3}\right)&=\alpha.\notag
\end{align}

\subsection{Radial solution: three equal stretches}

When we are looking for a radial solution $(\lambda_1,\lambda_2,\lambda_3)=(\beta^+,\beta^+,\beta^+)$ of the system \eqref{pbs1}, we are looking for a solution of the equation
\begin{align}
T_1(\beta^+,\beta^+,\beta^+)\equiv\frac{(3 M-2) \,(\beta^+)^6+6 \,(\beta^+)^2-4-3 \,M}{6 \,\beta^+}=\alpha.
\end{align}

As shown, such a solution exists, and its uniqueness is equivalent to the conditions on $h_{\rm CG}$ from Proposition~\ref{propNH}. For the Neo-Hooke-Ciarlet-Geymonat model, condition \eqref{propc1cl} is
\begin{align}\label{proplogneo2}
 \frac{(3 M-2) \left(5 x^2+7\right)}{216 \,x\,\sqrt[3]{x}}+\frac{1}{3\, \sqrt[3]{x^2}}>0 \qquad \forall\, x>0,
\end{align}
which is clearly satisfied for $M=\frac{\lambda+\frac{2\,\mu}{3}}{\mu }>\frac{2}{3}$. The other two conditions \eqref{propc1cl}$_{2,3}$ are also satisfied, since
\begin{align}
\lim_{x\rightarrow 0}\left(\frac{(3 M-2) \sqrt[3]{x^2} \left(x^2-7\right)}{72 \,x}+\sqrt[3]{x}\right)=-\infty \qquad \text{and}\qquad \lim_{x\rightarrow \infty}\left(\frac{(3 M-2) \sqrt[3]{x^2} \left(x^2-7\right)}{72\, x}+\sqrt[3]{x}\right)=\infty.
\end{align}
 Therefore, the corresponding radial solutions are unique.

\restylefloat*{figure}

\begin{figure}[!h]
	\centering
	\begin{minipage}{.5\textwidth}
		\centering
		\includegraphics[width=0.8\linewidth]{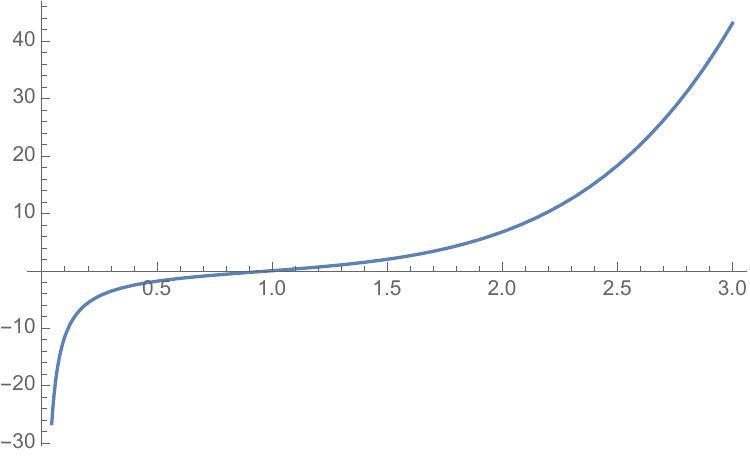}
		\caption{The plot of $\beta^+\mapsto f_{\rm Biot}(\beta^+)=T_1(\beta^+,\beta^+,\beta^+)$}
		\label{f-biot}
	\end{minipage}
	\begin{minipage}{0.5\textwidth}
		\centering
		\vspace*{1.6cm}
		\includegraphics[width=1\linewidth]{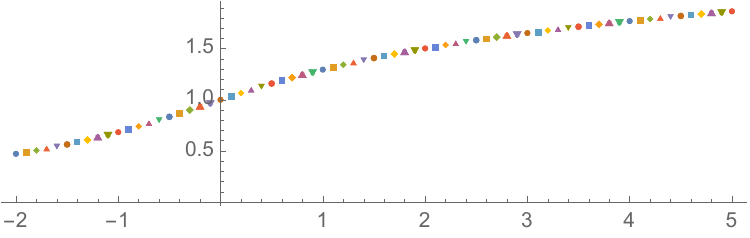}
		\captionsetup{labelsep=space,justification=justified,singlelinecheck=off}
		\caption{The solution $\beta^+$  of $T_1(\beta^+,\beta^+,\beta^+)=\alpha$ depends strictly monotone on $\alpha$.}
		\label{diag}
	\end{minipage}
\end{figure}

\begin{proposition}
	For the compressible Neo-Hooke-Ciarlet-Geymonat material, for all $\alpha\in \mathbb{R}$,  the constitutive equation  $
	T_{\rm Biot} (U)=\alpha\, \id
	$
	has a unique radial solution $U=\beta^+\cdot\id_3$, $\beta^+>0$.
	Moreover, since the mapping \begin{align}f_{\rm Biot}:(0,\infty)\to \mathbb{R},\qquad f_{\rm Biot}(\beta^+):=T_1(\beta^+,\beta^+,\beta^+)=\frac{(3 M-2) \,(\beta^+)^6+6 \,(\beta^+)^2-4-3 \,M}{6 \,\beta^+}\end{align} is strictly monotone  increasing, continuous and surjective (see Figure \ref{f-biot}),  the solution $\beta^+=\beta^+(\alpha)$ is a monotone  increasing function (see Figure \ref{diag}).
\end{proposition}

However, when the bifurcation problem is studied in the Rivlin's cube problem, we are interested to study if for all $\alpha>0$ all  radial solutions $U=\beta^+\cdot \id$ of the equation 
\begin{align}\label{eqNHcl1}
T_{\rm Biot}^{\rm NH} (U)=\alpha\, \id
\end{align}
are locally unique in the general classes of all possible solutions $U\in {\rm Sym}^{++}(3)$ (possibly non-radial), see Table \ref{tablec} for a summary of the constitutive conditions used in this paper. 

Note that we are not interested  in having a unique solution of \eqref{eqNHcl1}, but a locally unique solution. This is  because we have to study whether the solution may continuously (in the sense of the continuity of the map $\alpha\mapsto U(\alpha)$) depart from a radial one to a non-radial one and vice-versa. This is only possible in those points in which the mapping $U\mapsto T_{\rm Biot}(U)$ is not invertible, i.e., using Theorem \ref{theorem:invertibility}, in those points $(\lambda_1^*,\lambda_2^*, \lambda_3^*)$ where 
\begin{align}
\det {\rm D}\widehat{T}(\lambda_1^*,\lambda_2^*, \lambda_3^*)=0.
\end{align}

Specifically, we are thus interested in the existence of a radial $U=\lambda_1\, \id$, such that
\begin{align}
\det {\rm D}\widehat{T}(\lambda_1,\lambda_1, \lambda_1)=0,
\end{align}
i.e., whether the map $U\mapsto T_{\rm Biot}(U)$ loses local differentiable invertibility in a radial $U$.

Indeed, for each material given by $M>\frac{2}{3}$ we have that 
\begin{align}
\det {\rm D}\widehat{T}(\lambda_1,\lambda_1, \lambda_1)=\frac{\left[(2-3 M)\, \lambda_1^6+6 \lambda_1^2+4+3 M\right]^2 \left[5\, (3 M-2)\, \lambda_1^6+6\, \lambda_1^2+4+3 M\right]}{216 \,\lambda_1^6},
\end{align}
and therefore, $U\mapsto T_{\rm Biot}(U)$ loses the local invertibility in $U=\lambda^*\, \id$, where $\lambda^*$ is the unique solution (see the proof of Proposition \ref{notinv}) of the equation
\begin{align}\label{losti1}
(2-3\, M) \,\lambda_1^6+6 \, \lambda_1^2+4+3 M=0. 
\end{align}
Since for $M>\frac{2}{3}$ the above equation has a unique positive solution, see the proof of Proposition \ref{notinv} and Figure \ref{xm} (for fixed $M$, the unique  solution is the intersection of the red line with the blue curve), we  argue that the bifurcation occurs for all admissible constitutive parameters in only one radial solution. 

We recall that, from the proof of Proposition \ref{thmon}, we have 
	\begin{align}
	m^{\rm Biot}_1(\lambda_1,\lambda_1, \lambda_1)>0,\qquad 
	m^{\rm Biot}_2(\lambda_1,\lambda_1, \lambda_1)>0,\qquad 
	m^{\rm Biot}_3(\lambda_1,\lambda_1, \lambda_1)>0, \qquad \text{for}\qquad \lambda_1<\lambda^*,\notag\\
	m^{\rm Biot}_1(\lambda_1,\lambda_1, \lambda_1)>0,\qquad 
	m^{\rm Biot}_2(\lambda_1,\lambda_1, \lambda_1)=0,\qquad 
	m^{\rm Biot}_3(\lambda_1,\lambda_1, \lambda_1)=0, \qquad \text{for}\qquad \lambda_1=\lambda^*,\\
	m^{\rm Biot}_1(\lambda_1,\lambda_1, \lambda_1)>0,\qquad 
	m^{\rm Biot}_2(\lambda_1,\lambda_1, \lambda_1)<0,\qquad 
	m^{\rm Biot}_3(\lambda_1,\lambda_1, \lambda_1)>0, \qquad \text{for}\qquad \lambda_1>\lambda^*,\notag
	\end{align}
where $	m^{\rm Biot}_1(\lambda_1,\lambda_1, \lambda_1), 	m^{\rm Biot}_2(\lambda_1,\lambda_1, \lambda_1)$ and $	m^{\rm Biot}_1(\lambda_1,\lambda_1, \lambda_1)$	are the principal minors of ${\rm D}\widehat{T}(\lambda_1,\lambda_1, \lambda_1)$. 

\begin{table}[h!]
	\centering
	\resizebox{17cm}{!}{	\begin{tabular}{|l|l|l|l|l|}
			\hline
			{\bf In terms of }	&{\bf invertibility of $T_{\rm Biot}$}& {\bf \textcolor{magenta}{Hilbert-monotonicity of $T_{\rm Biot}$}}&  {\bf \textcolor{blue}{strong-monotonicity of $\widehat{T}$}}&  {\bf \textcolor{orange}{energetic stability}}\\
			\hline
			\begin{minipage}{4cm}the principal Biot stresses $\widehat{T}=(T_1,T_2,T_3)^T$\end{minipage}&	\begin{minipage}{4.9cm}\smallskip
				${\rm D}\widehat{T}\,(\lambda_1,\lambda_2,\lambda_3)$  is invertible\\ for any $(\lambda_1,\lambda_2,\lambda_3)\in \mathbb{R}^3_+$;
				\\
				and\\
				$\|\widehat{T}(\lambda_1,\lambda_2,\lambda_3)\|_{\mathbb{R}^3}\to \infty$ as \\ $\|(\log\lambda_1,\log\lambda_2,\log\lambda_3)\|_{\mathbb{R}^3}\to \infty$,
				\smallskip
			\end{minipage}&\begin{minipage}{4.3cm}
				\textcolor{magenta}{${\rm D}\widehat{T}\,(\lambda_1,\lambda_2,\lambda_3)\in{\rm Sym}^{+}(3)$}
			\end{minipage}&\begin{minipage}{4.3cm}
			\textcolor{blue}{${\rm D}\widehat{T}\,(\lambda_1,\lambda_2,\lambda_3)\in{\rm Sym}^{++}(3)$}
		\end{minipage}& \begin{minipage}{4.1cm}
				\textcolor{orange}{	$\frac{T_i-\epsilon_{ij} \,T_j}{\lambda_i-\epsilon_{ij} \,\lambda_j}\geq 0  \ i\neq j\  \text{no sum}
					$}\\
				\textcolor{orange}{and}  \\
				\textcolor{orange}{${\rm D} \widehat{T}(\lambda_1,\lambda_2,\lambda_3)\in{\rm Sym}^{+}(3)$.}\end{minipage}\\
			\hline
			\begin{minipage}{3.5cm}the energy expressed in the principal stretches\\
				$W(F)=g(\lambda_1,\lambda_2,\lambda_3)$\end{minipage}& \begin{minipage}{4.9 cm}\smallskip
				${\rm D}^2g\,(\lambda_1,\lambda_2,\lambda_3)$  is invertible\\ for any $(\lambda_1,\lambda_2,\lambda_3)\in \mathbb{R}^3_+$;\\
				and\\
				$\|{\rm D}g(\lambda_1,\lambda_2,\lambda_3)\|_{\mathbb{R}^3}\to \infty$ as \\
				$\|(\log\lambda_1,\log\lambda_2,\log\lambda_3)\|_{\mathbb{R}^3}\to \infty$,\smallskip
			\end{minipage}&\begin{minipage}{4.4cm}
				\textcolor{magenta}{	${\rm D}^2g\,(\lambda_1,\lambda_2,\lambda_3)\in{\rm Sym}^{+}(3)$}
			\end{minipage}&\begin{minipage}{4.4cm}
			\textcolor{blue}{	${\rm D}^2g\,(\lambda_1,\lambda_2,\lambda_3)\in{\rm Sym}^{++}(3)$}
		\end{minipage}& \begin{minipage}{4.6cm}
				\textcolor{orange}{$\frac{\frac{\partial g}{\partial \lambda_i}-\epsilon_{ij} \,\frac{\partial g}{\partial \lambda_j}}{\lambda_i-\epsilon_{ij} \,\lambda_j}\geq 0, \ i\neq j \  \text{no sum}
					$}\\
				\textcolor{orange}{and} \\
				\textcolor{orange}{${\rm D}^2 g(\lambda_1,\lambda_2,\lambda_3)\in{\rm Sym}^{+}(3)$.}\end{minipage}\\
			\hline
	\end{tabular}}	\caption{A summary of the constitutive conditions used in this paper. Here, $\epsilon_{ij}
		=1$ if $\{i,j\}=\{1,2\}$  or  $\{2,3\}$ or $\{3,1\}$, $\epsilon_{ij}=-1$ otherwise.}\label{tablec}
\end{table}

Hence, even if the map $\alpha\mapsto \beta(\alpha)$ giving the solution of $T_{\rm Biot}(\beta\,\id )=\alpha\,\id $ is strictly monotone  increasing, the relation $T_{\rm Biot}=T_{\rm Biot}(U)$ could be locally strictly monotone  only at those radial $U=\lambda_1\, \id$  for which $\lambda_1<\lambda^*$, and it loses its strict monotonicity on those radial $U=\lambda_1\, \id$ for which $\lambda_1>\lambda^*$, see Figures \ref{mr1} and \ref{mr2}. Moreover, $\widehat{T}$ is strongly monotone only for $\lambda_1<\lambda^*$. This is unphysical, since for purely radial deformations, the Biot stress should clearly increase with the strain; therefore, for $U=\lambda_1^*\, \id$, $\lambda_1\geq\lambda^*$, the radial solution should not be considered physically admissible anymore. In other words,  the cube cannot remain a cube by increasing its length above $\lambda^*$ and at the same time keeping the strict monotonicity of the $T_{\rm Biot}=T_{\rm Biot}(U)$ relation.

\begin{figure}[h!]
	\begin{minipage}{.3\textwidth}
		\centering
		\includegraphics[width=1\linewidth]{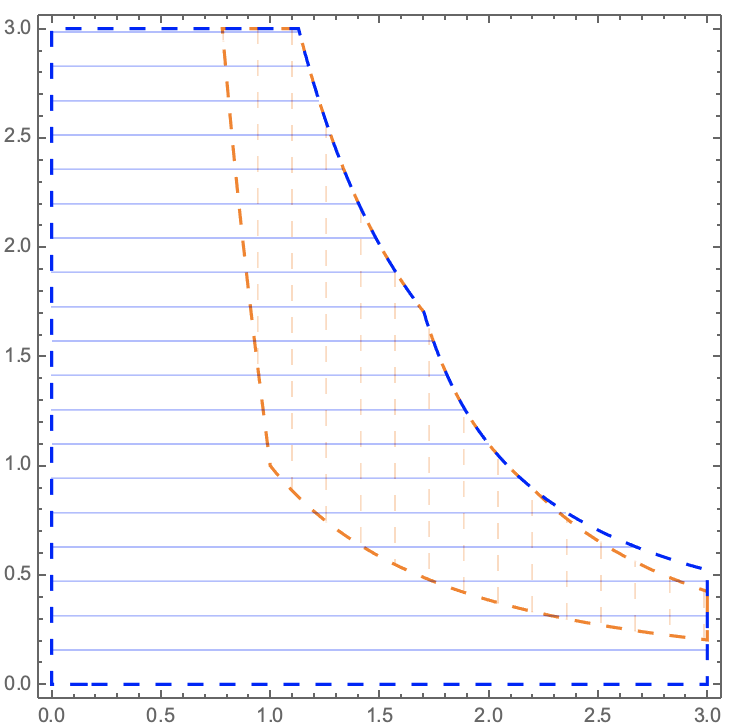}
		\caption{For $M=1$, the strong monotonicity region in the points $(\lambda_1,\lambda_1,\lambda_2)$ (the blue region) versus the energetic stability region in the points $(\lambda_1,\lambda_1,\lambda_2)$ (the orange region).}\label{mr1}
	\end{minipage}\quad 
	\begin{minipage}{.3\textwidth}
		\centering
		\vspace*{-0.7cm}
		\includegraphics[width=1\linewidth]{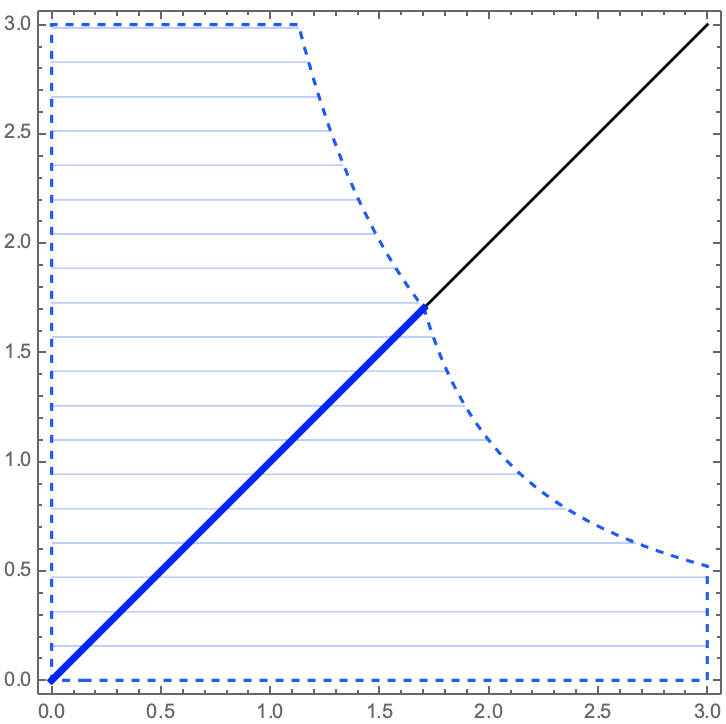}
		\caption{ The strong monotonicity is satisfied on the radial solutions until the bifurcation point  (on the blue curve), i.e., $0<\lambda_1\leq\lambda^*$. }\label{mr2}
	\end{minipage}\quad 
\begin{minipage}{.3\textwidth}
	\centering
	\includegraphics[width=1\linewidth]{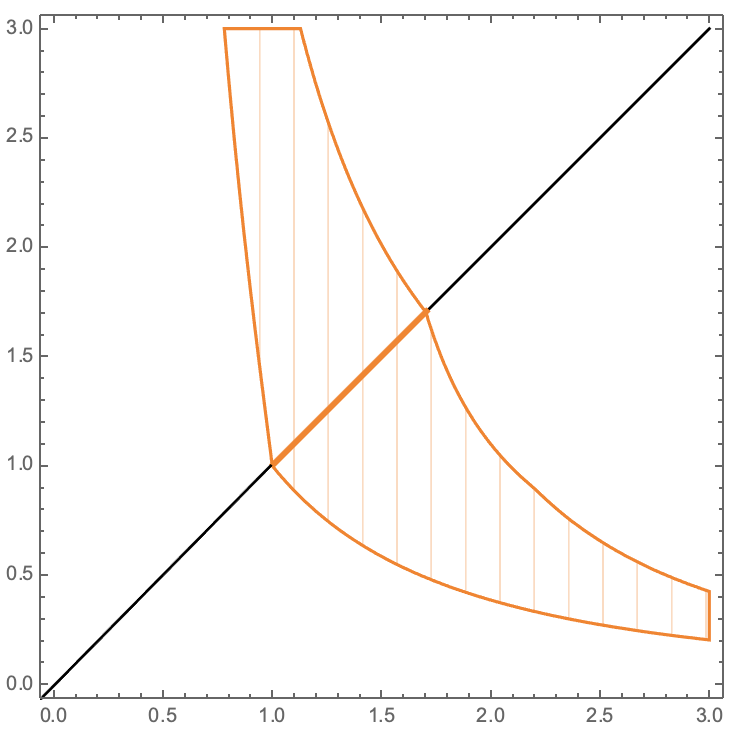}
	\caption{ Contrary to strong  monotonicity,  the energetic stability  is satisfied on the radial solutions only starting with $1$ and until the bifurcation point  (on the blue curve), i.e., $1\leq\lambda_1\leq \lambda^*$. }\label{mr3}
\end{minipage}
\end{figure}
Regarding the energetic stability of the radial solution, we note that the stability conditions \eqref{staine1} against arbitrary perturbations   for radial solutions, by letting $\lambda_i\to\lambda$, read as
\begin{align}
&\left(\frac{\partial^2 g}{\partial \lambda_1^2}-\frac{\partial^2 g}{\partial \lambda_1\partial \lambda_2}\right)\Big|_{\lambda_i=\lambda}\geq 0, \qquad \left(\frac{\frac{\partial g}{\partial \lambda_2}+\frac{\partial g}{\partial \lambda_1}}{2\lambda_1}\right)\Big|_{\lambda_i=\lambda}\geq 0.
\end{align}
In addition to these conditions, energetic stability requires to check the positive semi-definiteness of the Hessian matrix evaluated in the radial solution ${\rm D}^2 g=\left(\frac{\partial^2 g}{\partial \lambda_i\partial \lambda_j}\right)\Big|_{\lambda_i=\lambda}$, too.

The first  inequality is equivalent to 
\begin{align}
 \frac{(2-3 M) \lambda_1^6+6\, \lambda_1^2+{3 M+4}}{6\,\lambda_1^2}\geq0,
\end{align}
while the second one is equivalent to
\begin{align}
\frac{(3 M-2)  \lambda_1^6+6  \lambda_1^2-4-3 M}{3\,  \lambda_1^2}\geq0.
\end{align} 
Note that the Hessian matrix ${\rm D}^2 g$ is actually $
{\rm D}\widehat{T}(\lambda_1,\lambda_1, \lambda_1)$, and therefore the last condition for stability, i.e. the positive semi-definiteness of  ${\rm D}^2 g$, is implied by the strict monotonicity of $\widehat{T}$. Moreover, the first  inequality required by the stability criterion is redundant, {and so} it follows from the local positive definiteness of $
{\rm D}\widehat{T}(\lambda_1,\lambda_1, \lambda_1)$, see the expression of $m_1^{\rm Biot}(\lambda_1,\lambda_1, \lambda_1)$. In conclusion, the stability of the radial solutions is implied by  the strong monotonicity of $\widehat{T}$ in the radial solutions, which holds true only until the radial solution reaches the bifurcation point, i.e. for $\lambda_i\leq \lambda^*$, $i=1,2,3$, together {with} 
\begin{align}\label{bneg}
(3 M-2)  \lambda_1^6+6  \lambda_1^2-4-3 M\geq0.
\end{align}

Since, $T_{1}(\lambda_1,\lambda_1,\lambda_1)=\frac{(3 M-2) \lambda_1^6+6 \lambda_1^2-4-3 M}{6 \lambda_1}$, \eqref{bneg} is possible only when $T_{\rm Biot}(\diag(\lambda_1,\lambda_1,\lambda_1))=\alpha\id$, with $\alpha\geq 0$. Note again that $\lambda_1\mapsto T_{1}(\diag(\lambda_1,\lambda_1,\lambda_1))$, $\lambda_1>0$ is strictly monotone  increasing, continuous and surjective, $\lim_{\lambda_1\to \infty}  T_{1}(\lambda_1,\lambda_1,\lambda_1)=\infty$, $\lim_{\lambda_1\to -\infty}  T_{1}(\lambda_1,\lambda_1,\lambda_1)=-\infty$ and $  T_{1}(1,1,1)=0$. Hence, inequality \eqref{bneg} holds true only for $\lambda_1\geq1$, i.e., when the radial solution is a uniform extension.

Hence, the radial solution is stable, see Figure \ref{mr3}, only for
\begin{align}
1\leq\lambda_1\leq\lambda^*,
\end{align}
which lets us conclude that the stability criteria for the radial solutions are more restrictive than the monotonicity criteria.

\subsection{Two equal principal stretches $(\lambda_1= \lambda_2\neq\lambda_3)$}

In this subsection we find the solutions of the form $U=\begin{pmatrix}
\lambda_1&0&0\\
0&\lambda_1&0\\
0&0&\lambda_2
\end{pmatrix}$, $\lambda_1\neq \lambda_2$, of the equation $T_{\rm Biot}(U)=\alpha\, \id$. We show that in the neighbourhood of the unique radial   solution $U=\lambda_1^*\, \id$ for which $\det {\rm D}\widehat{T}(\lambda_1^*,\lambda_1^*,\lambda_1^*)$ is not invertible, i.e., in a neighbourhood of   $\alpha^*=T_1(\lambda_1^*,\lambda_1^*,\lambda_1^*)$, the equation $T_{\rm Biot}(U)=\alpha_\epsilon\, \id$ admits a solution $U_\epsilon=\begin{pmatrix}
\lambda_1^\varepsilon&0&0\\
0&\lambda_1^\varepsilon&0\\
0&0&\lambda_2^\varepsilon
\end{pmatrix}$, $\lambda_1^\varepsilon\neq \lambda_2^\varepsilon$ which tends to the radial solution $U^*=\begin{pmatrix}
\lambda_1^*&0&0\\
0&\lambda_1^*&0\\
0&0&\lambda_1^*
\end{pmatrix}$ when $\alpha_\varepsilon$ goes to $\alpha^*$.

\smallskip\noindent
Hence, we have to solve the system
\begin{align}
\frac{\frac{1}{6} (3 M-2) \left(\lambda_1^4 \lambda_2^2-1\right)+\lambda_1^2-1}{\lambda_1}&=\alpha,\\
\frac{1}{6} (3 M-2) \lambda_1^4 \lambda_2-\frac{3 M+4}{6 \lambda_2}+\lambda_2&=\alpha,\notag
\end{align}
since $T_1(\beta, \beta, \gamma)=T_2(\beta, \beta, \gamma)=\frac{\frac{1}{6} (3 M-2) \left(\beta^4 \gamma^2-1\right)+\beta^2-1}{\beta}$ and $T_3(\beta, \beta, \gamma)=\frac{1}{6} (3 M-2) \beta^4 \gamma-\frac{3 M+4}{6 \gamma}+\gamma$, or equivalently to solve the system
\begin{align}\label{de}
T_1(\lambda_1, \lambda_1, \lambda_2)-T_3(\lambda_1, \lambda_1, \lambda_2)
\equiv-\frac{(\lambda_1-\lambda_2) \left((3 M-2) \lambda_1^4 \lambda_2^2-3 M-6 \lambda_1 \lambda_2-4\right)}{6 \lambda_1 \lambda_2}&=0,\\
\frac{\frac{1}{6} (3 M-2) \left(\lambda_1^4 \lambda_2^2-1\right)+\lambda_1^2-1}{\lambda_1}&=\alpha.\notag
\end{align}
Hence, $\lambda_1=\lambda_2$, i.e., {yielding} the radial solution obtained in the previous section, or 
\begin{align}
\lambda_2=\frac{\sqrt{(9 M^2 +6 M -8) \lambda_1^4+9 \lambda_1^2}+3 \lambda_1}{(3 M -2) \lambda_1^4}, 
\end{align}
with $\lambda_1$ being a solution of 
\begin{align}\label{de2}
\frac{\sqrt{\left(9 M^2+6 M-8\right) \lambda_1^4+9 \lambda_1^2}+(3 M-2) \lambda_1^5+3 \lambda_1}{(3 M-2) \lambda_1^4}=\alpha.
\end{align}
For any $M>\frac{2}{3}$, the function 
\begin{align}
\ell:(0,\infty)\to (0,\infty), \quad \ell(\lambda_1)=\frac{\sqrt{\left(9 M^2+6 M-8\right) \lambda_1^4+9 \lambda_1^2}+(3 M-2) \lambda_1^5+3 \lambda_1}{(3 M-2) \lambda_1^4}
\end{align}
is convex.
Moreover, we have 
\begin{align}
\lim_{\lambda_1\to 0}\ell(\lambda_1)=\infty=\lim_{\lambda_1\to \infty}\ell(\lambda_1).
\end{align}
Thus,
\begin{itemize}
	\item for $\alpha<\min_{\lambda_1>0}\ell(\lambda_1)$ the equation \label{de2} has no solutions. Therefore, only the radial solution which always exists is a solution of the equation $T_{\rm Biot}(U)=\alpha\, \id$. 
	\item for $\alpha=\min_{\lambda_1>0}\ell(\lambda_1)$ the equation \label{de2} has one solution. Hence, the equation $T_{\rm Biot}(U)=\alpha\, \id$ has two solutions: one radial and another with two equal eigenvalues.
	\item  for all $\alpha>\min_{\lambda_1>0}\ell(\lambda_1)$ the equation \label{de2} has two different admissible solutions, which lead to two non-radial admissible solutions of the equation $T_{\rm Biot}(U)=\alpha\, \id$.
Besides these two solutions we have the already found radial solution, too.
\end{itemize}

\begin{figure}[h!]
	\centering
	\includegraphics[scale=0.75]{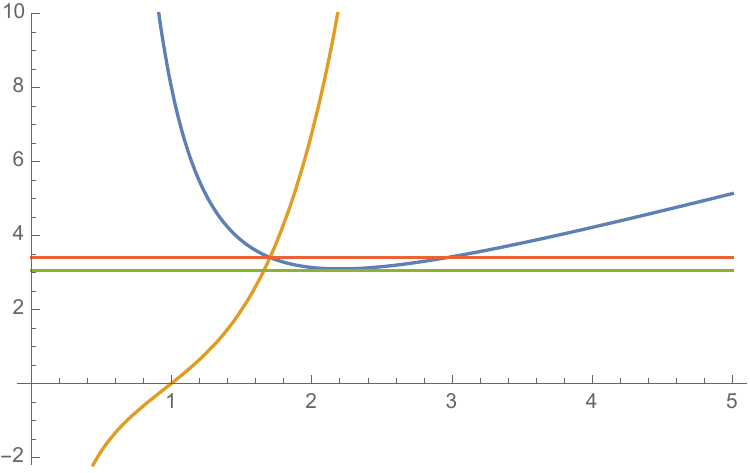}
	\caption{\footnotesize{The plot of $\beta\mapsto \ell(\beta)$ (blue curve) and plot of $\beta\mapsto f_{\rm Biot}(\beta)$ (orange curve) for $M=1$.}}
	\label{2e}
\end{figure}%

More about the behaviour of these solutions may be observed from Figure \ref{2e}, using $M=1$. For values of $\alpha$ below the green line, there exists only the radial solution. Between the green and the red lines there exist the radial solution and  two other non-radial solutions with two equal eigenvalues. By approaching the red line, one non-radial solutions goes to the radial solution situated at the intersection point of the orange curve with the blue curve, while the other non-radial solution tends into the other direction of the blue curve and will never be  in the neighbourhood of a radial solution. For values of $\alpha$ above the red line, there are again three different solutions. 

It is easy to find that no bifurcation occurs in compression, since for $\alpha<0$ there exists only the radial solution. Thus, at the radial solution given by the intersection point of the blue and orange curve,  the relation $T_{\rm Biot}=T_{\rm Biot}(U)$ is not locally invertible, since the radial solution is not locally unique. 

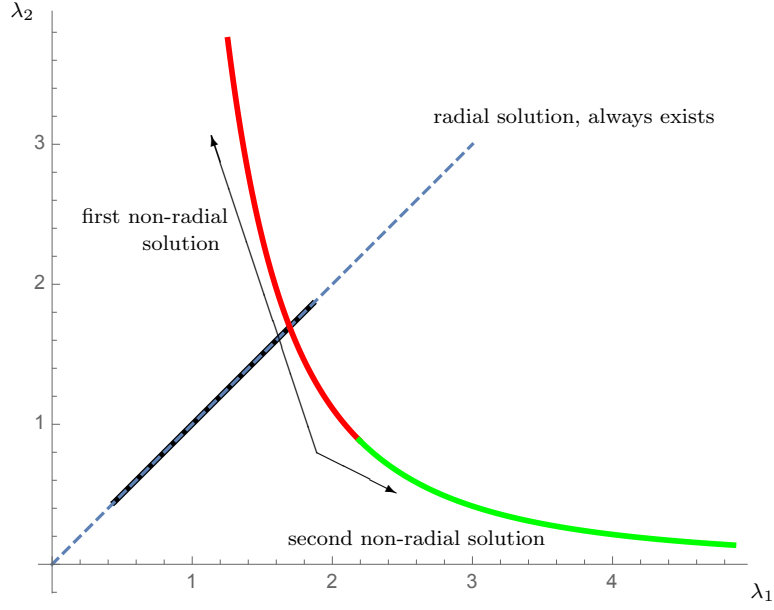
\begin{figure}[h!]
	\centering
	\includegraphics[scale=0.75]{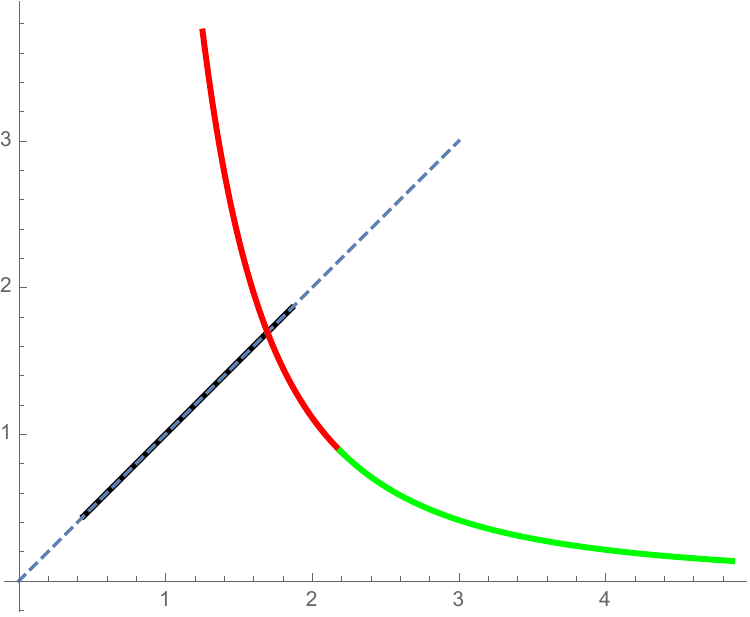}
	\put(-120,180){\footnotesize radial solution, always exists}
	\put(-130,80){\footnotesize }
	\put(-253,142){\footnotesize first non-radial}
		\put(-230,132){\footnotesize  solution}
	\put(-164,55){\vector(2,-1){30}}
	\put(-164,55){\vector(-1,3){40}}
		\put(-175,20){\footnotesize second non-radial solution}
		\put(0,0){\footnotesize $\lambda_1$}
		\put(-280,220){\footnotesize $\lambda_2$}
	\caption{\footnotesize{ For $M=1$, radial and non-radial solutions, bifurcation.}}\label{bif1}
	
\end{figure}%

In conclusion, a bifurcation point  $\lambda^*$ is a solution of \eqref{losti1}. In Figure \ref{bif1} we plot the path of the point $(\lambda_1,\lambda_2)$ as a function of $\alpha\in [0,5]$ with a step size of $0.1$, by solving numerically the equation \eqref{de2} for $M=1$, but the analysis is completely the same for any other value of $M$.

\begin{figure}[h!]
	\centering
	\begin{minipage}{.3\textwidth}
		\centering
		\includegraphics[width=1\linewidth]{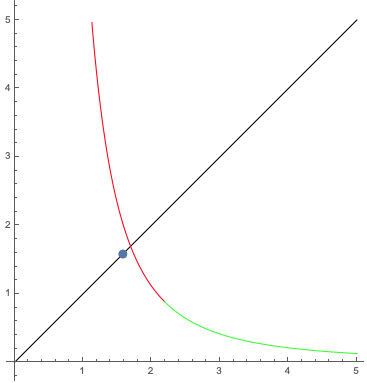}
		\put(-75,135){$\alpha=2.5$}
			\put(0,0){\footnotesize $\lambda_1$}
		\put(-153,150){\footnotesize $\lambda_2$}
	\end{minipage}\qquad 
	\begin{minipage}{.3\textwidth}
	\centering
		\includegraphics[width=1\linewidth]{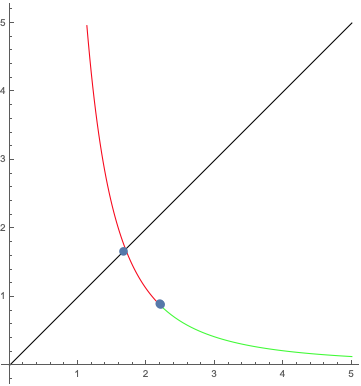}
		\put(-75,135){$\alpha=2.8$}
			\put(-90,50){$\lambda_1=1.62561$}
				\put(0,0){\footnotesize $\lambda_1$}
			\put(-153,150){\footnotesize $\lambda_2$}
		\captionsetup{labelsep=space,justification=justified,singlelinecheck=off}
	\end{minipage}\qquad 
	\begin{minipage}{.3\textwidth}
		\centering
	\includegraphics[width=1\linewidth]{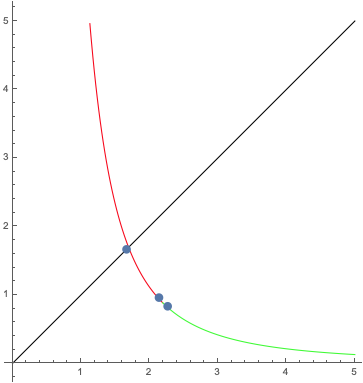}
	\put(-75,135){$\alpha=3.09675$}
		\put(0,0){\footnotesize $\lambda_1$}
	\put(-153,150){\footnotesize $\lambda_2$}
	\captionsetup{labelsep=space,justification=justified,singlelinecheck=off}
\end{minipage}\vspace{5mm}\\
	\centering
\begin{minipage}{.3\textwidth}
	\centering
	\includegraphics[width=1\linewidth]{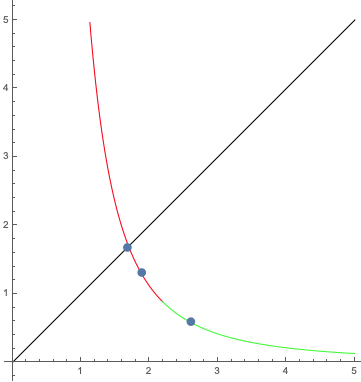}
	\put(-75,135){$\alpha=3.1$}
		\put(0,0){\footnotesize $\lambda_1$}
	\put(-153,150){\footnotesize $\lambda_2$}
\end{minipage}\qquad 
\begin{minipage}{.3\textwidth}
\centering
	\includegraphics[width=1\linewidth]{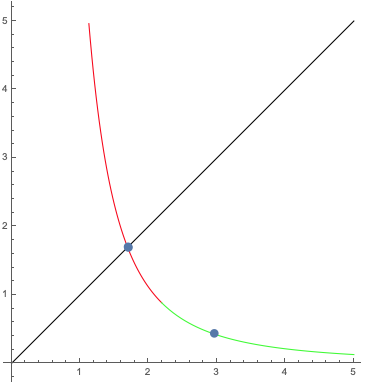}
	\put(-75,135){$\alpha=3.4$}
		\put(-90,50){$\lambda_1=1.70237$}
			\put(0,0){\footnotesize $\lambda_1$}
		\put(-153,150){\footnotesize $\lambda_2$}
	\captionsetup{labelsep=space,justification=justified,singlelinecheck=off}
\end{minipage}\qquad 
\begin{minipage}{.3\textwidth}
	\centering
	\includegraphics[width=1\linewidth]{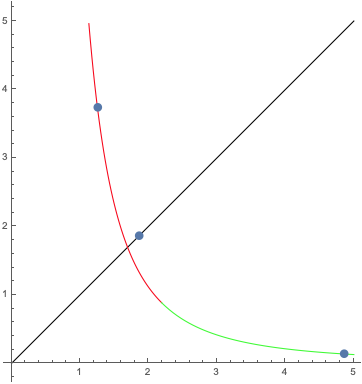}
	\put(-75,135){$\alpha=5$}
		\put(0,0){\footnotesize $\lambda_1$}
	\put(-153,150){\footnotesize $\lambda_2$}
\end{minipage}
\caption{For $M=1$, the solutions of the equation $T_i(\lambda_1,\lambda_1,\lambda_2)=\alpha, i=1,2,3$ for a sequence of increasing values of $\alpha$.}\label{film}
\end{figure}

\begin{figure}[h!]
\hspace*{2.5cm}
	\includegraphics[scale=0.5]{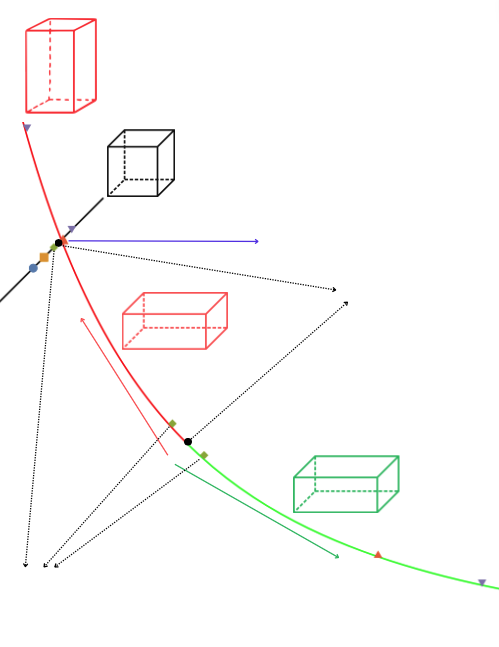}
	\put(-155,240){\footnotesize radial solution}
	\put(-130,80){\footnotesize }
	\put(-180,140){\footnotesize first non-radial}
		\put(-240,320){\footnotesize first non-radial}
	\put(-105,100){\footnotesize second non-radial solution}
		\put(-270,30){\footnotesize three solutions: one radial solution }
			\put(-207,20){\footnotesize and two non-radial solutions}
			\put(-270,10){\footnotesize - one non-radial  solution goes up on the red curve and it reaches the bifurcation point}
			\put(-270,0){\footnotesize - another radial solution goes down on the green curve and it never reaches the bifurcation point}
				\put(-70,177){\footnotesize two solutions: one radial solution  }
					\put(-70,167){\footnotesize  and one non-radial solution.}
						\put(-70,157){\footnotesize  At this magnitude $\alpha$ of the load }
							\put(-70,147){\footnotesize   ({\bf before} the bifurcation point see the black point marker)}
							\put(-70,137){\footnotesize  the non-radial solution appears.}
								\put(-119,212){\footnotesize \bf bifurcation point: a non-radial solution goes  }
									\put(-119,202){\footnotesize 
									 \bf continuously to a radial solution  and then becomes  }
								 	\put(-119,192){\footnotesize 
								 	\bf again non-radial as the magnitude $\alpha$ of the load increases.}
	\caption{\footnotesize{ For $M=1$, a zoom commented picture of radial and non-radial solutions, bifurcation.}}\label{bif1}
	
\end{figure}%

However, for one  branch of non-radial solutions (green curve) there is no value of the Biot-stress magnitude for which the cube may continuously switch from the radial solution  to a non-radial solution, while for the  other branch of non-radial solutions (red curve) there is a unique value of the Biot-stress magnitude for which the cube may continuously switch from the radial solution to a non-radial solution, see Figures \ref{film} and \ref{bif1}.

Regarding the strong monotonicity of the non-radial solutions, we remark that while the principal minor 
\begin{align}
m^{\rm Biot}_1(\lambda_1,\lambda_1, \lambda_2)&:=\left(\frac{M}{4}-\frac{1}{6}\right) \left(2 \lambda_1^2 \lambda_2^2+\frac{2}{\lambda_1^2}\right)+\frac{1}{\lambda_1^2}+1>0
\end{align} 
of  $
{\rm D} \widehat{T}(\lambda_1,\lambda_1, \lambda_2)$  is positive (not only on this curve), the second principal minor
\begin{align}
m^{\rm Biot}_2(\lambda_1,\lambda_1, \lambda_2):=-\frac{\left((3 M-2) \lambda_1^4 \lambda_2^2-3 M-6 \lambda_1^2-4\right) \left(3 (3 M-2) \lambda_1^4 \lambda_2^2+3 M+6 \lambda_1^2+4\right)}{36 \lambda_1^4},
\end{align}
is strictly positive only on those points on this curve
for which
\begin{align}
(2-3 M) \lambda_1^4 \lambda_2^2+3 M+6 \lambda_1^2+4>0,
\end{align}
i.e., for $\lambda_1$ satisfying
\begin{align}
-&\frac{6 \left(-\sqrt{\left(9 M^2+6 M-8\right) \lambda_1^2+9 }+(3 M-2) \lambda_1^4-3 \right)}{(3 M-2) \lambda_1^2}>0\\ &\iff\ \ \left(\lambda_1^4>\frac{3}{3M-2} \qquad  \text{and}\qquad -4 - 3 M - 6 \lambda_1^2 +( 3 M-2) \lambda_1^6>0\right).\notag
\end{align}
In fact, each $\lambda_1$ such that $4 - 3 M - 6 \lambda_1^2 +( 3 M-2) \lambda_1^6>0$ satisfies $\lambda_1^4>\frac{3}{3M-2}$, too. Indeed, if  $\lambda_1^4\geq \frac{3}{3M-2}$ then $-4 - 3 M - 6 \lambda_1^2 +( 3 M-2) \lambda_1^6\leq -4 - 3 M - 3 \lambda_1^2 <0$. Therefore, a necessary condition for the strong monotonicity of a non-radial solution is $\lambda_1>\lambda^*$. In Figure \ref{bif1} and Figure \ref{nrm1}, this means the part of the red curve below the $\lambda_1=\lambda_2$ curve and the entire green curve.  The analytic  study of the sign of the third principal minor $m^{\rm Biot}_3(\lambda_1,\lambda_1, \lambda_2)$ of $
{\rm D}\widehat{T}(\lambda_1,\lambda_1, \lambda_2)$ is more complicated. However, the numerical testing has shown that on the entire red curve the strong monotonicity of the principal Biot stresses vector is lost, while on the green curve the monotonicity holds true, see Figure \ref{nrm2}.

Numerical computations show that,  before and after the bifurcation point  the values of the internal energy density $W_{\rm CG}(F)$  as well as the absolute value of the total energy \eqref{totale} are smaller on the radial solutions, in comparison to the non-radial solutions, while this is not true for the total energy \eqref{totale}, even before the bifurcation. Note that the total energy is positive for contraction and negative for extension.

\begin{figure}[t]
	\begin{minipage}{.3\textwidth}
		\centering
		\includegraphics[width=1\linewidth]{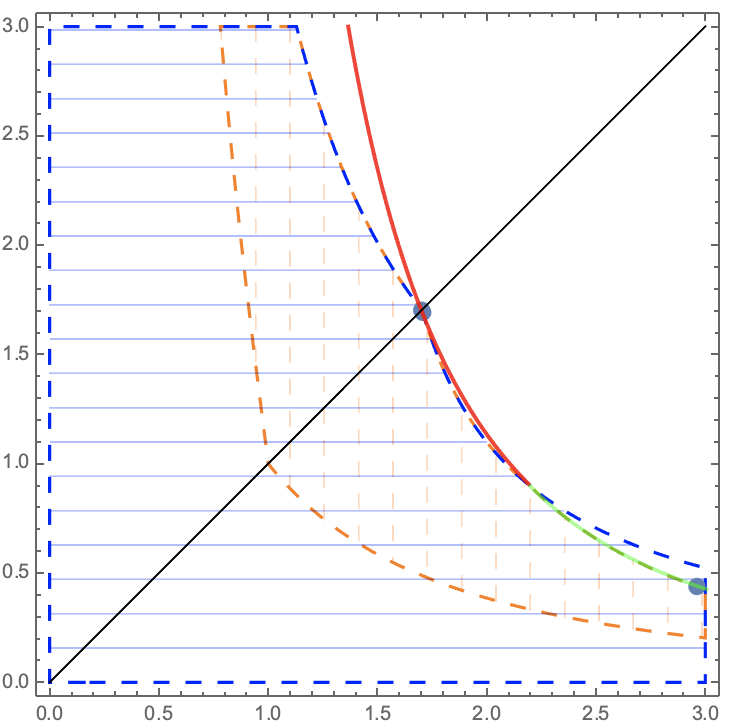}
			\caption{For $M=1$, the strong monotonicity region in the points $(\lambda_1,\lambda_1,\lambda_2)$ (the blue region) versus the energetic stability region in the points $(\lambda_1,\lambda_1,\lambda_2)$ (the orange region) together with the radial solutions (the black curve) and the non-radial solutions (the red curve and the green curve). Only one branch (the green cuve) of non-radial solutions belongs to the strong monotonicity domain. The same branch (the green curve) belongs to the energetic stability domain, too.  }\label{nrm1}
	\end{minipage}\quad 
	\begin{minipage}{.3\textwidth}
		\centering\vspace*{-0.8cm}
		\includegraphics[width=1\linewidth]{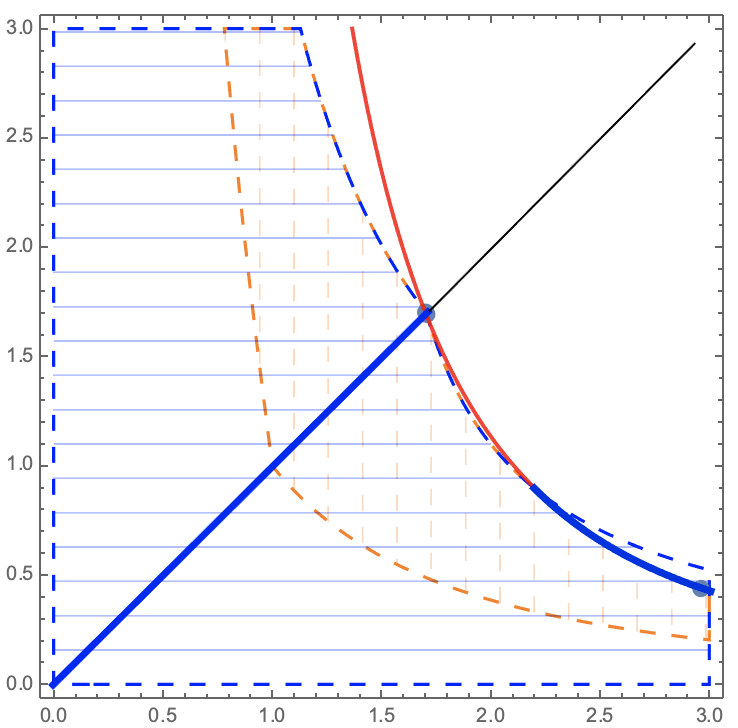}
		\caption{ The strong monotonicity is satisfied on the radial solutions until the bifurcation point is reached, while after the value of $\alpha$ for which bifurcation is present, one branch of the radial solutions preserves the strong monotonicity, while the radial solutions and the other branch (red curve) do not satisfy the strong monotonicity conditions. In this figure, the strong monotonicity is satisfied on the blue curves.}\label{nrm2}
	\end{minipage}\quad 
	\begin{minipage}{.3\textwidth}
		\centering\vspace*{-1.2cm}
		\includegraphics[width=1\linewidth]{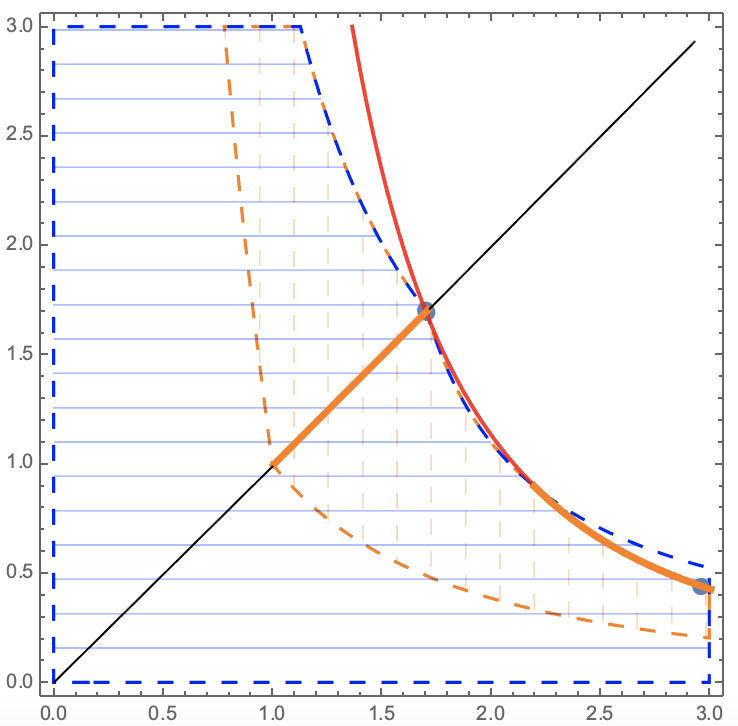}
		\caption{ The energetic stability  is  satisfied on the radial solutions only starting with $\lambda_1=1$ and until  the bifurcation point is attended. However, for non-radial solutions the situation is similar to the strong monotonicity, i.e., the stability is satisfied only on one branch of the radial solutions.  In this figure, the energetic stability is satisfied on the orange curves.}\label{nrm3}
	\end{minipage}

\end{figure}

In the following we discuss the stability of the non-radial solutions, the stability of the radial solutions being already discussed in the previous subsection. 
The stability conditions \eqref{snr} are equivalent to
\begin{align}
\frac{(2-3 M) \lambda_1^4 \lambda_2^2+3 M+6 \lambda_1^2+4}{6 \lambda_1^2}&\geq 0,\qquad\ 
\frac{(2-3 M) \lambda_1^4 \lambda_2^2+3 M+6 \lambda_1 \lambda_2+4}{6 \lambda_1 \lambda_2}\geq 0,\\
\frac{(3 M-2) \lambda_1^4 \lambda_2^2-3 M+6 \lambda_1^2-4}{6 \lambda_1^2}&\geq 0,\qquad 
\frac{1}{6} \left((3 M-2) \lambda_1^3 \lambda_2-\frac{3 M+4}{\lambda_1 \lambda_2}+6\right)\geq 0.\notag
\end{align}
The first two are equivalent to  the positivity of $m_2^{\rm Biot}(\lambda_1,\lambda_1,\lambda_2)$, so it implies $\lambda_1\geq \lambda^*$ while the third implies $\alpha\geq 0$ which is always satisfied since the non-radial solution is present only in extension. For $\lambda_1\geq \lambda^*$ it follows that the fourth inequality is satisfied, too.

The study of the positive semi-definiteness of ${\rm D}^2 g$ is similar  to the  study of the strong monotonicity of $\widehat{T}$ on the non-radial solution, which has already been treated above. 

Summarising, the energetic stability of the non-radial solutions is equivalent to the monotonicity of the map $U\mapsto T_{\rm Biot}(U)$ in these points. Therefore, the radial solutions are stable if and only if $1\geq \lambda_1\geq \lambda^*$, while the non-radial solutions are energetic stable only on the green branch of the non-radial solution. The energetic stable solutions are given in Figure \ref{nrm3} by the orange curve.

\subsection{Unequal principal stretches $(\lambda_i\neq \lambda_j, \ i\neq j$)}
\begin{figure}[h!]
	\centering
	\includegraphics[scale=0.75]{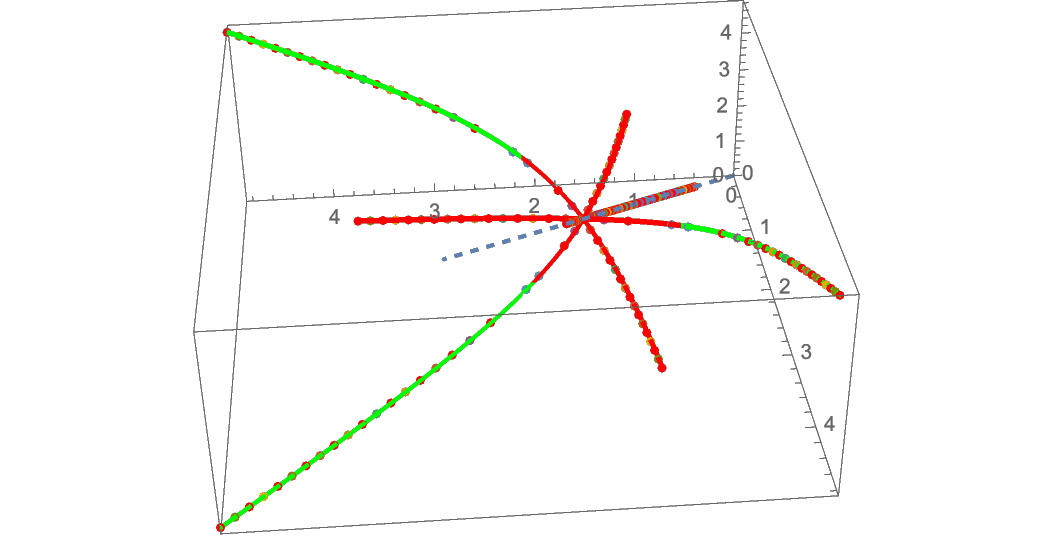}
	\caption{\footnotesize{The numerical simulation for $M=1$ of the solutions $(\lambda_1,\lambda_2,\lambda_3)$ of the equation $T_{\rm Biot}(U)=\alpha\, \id$,\  $\alpha\in[-2,5]$}}
	\label{bif2}
\end{figure}%

Using the expressions \eqref{pbs} of the principal Biot stresses
we find that the general solution $(\lambda_1,\lambda_2,\lambda_3)$ of the equation $T_{\rm Biot}(U)=\alpha\, \id$ is described by the following system
\begin{align}
-\frac{(\lambda_1-\lambda_3) \left((3 M-2) \lambda_1^2 \lambda_2^2 \lambda_3^2-3 M-6 \lambda_1 \lambda_3-4\right)}{6 \lambda_1 \lambda_3}=0,\notag\\-\frac{(\lambda_2-\lambda_3) \left((3 M-2) \lambda_1^2 \lambda_2^2 \lambda_3^2-3 M-6 \lambda_2 \lambda_3-4\right)}{6 \lambda_2 \lambda_3}=0,\\
\frac{1}{6} (3 M-2) \lambda_1^2 \lambda_2^2 \lambda_3-\frac{3 M+4}{6 \lambda_3}+\lambda_3=\alpha.\notag
\end{align}

If $\lambda_1=\lambda_3$ or $\lambda_2=\lambda_3$, then we are in the situation of the previous section.  If $\lambda_1\neq \lambda_2\neq \lambda_3\neq \lambda_1$ then 
\begin{align}
(3 M-2) \lambda_1^2 \lambda_2^2 \lambda_3^2-3 M-6 \lambda_1 \lambda_3-4=0 \ \ \text{and}\ \ \ (3 M-2) \lambda_1^2 \lambda_2^2 \lambda_3^2-3 M-6 \lambda_2 \lambda_3-4=0,
\end{align}
which implies that $\lambda_1=\lambda_2$. So the entire discussion reduces to the situation when two singular values are equal, a conclusion which may be observed  from the numerical simulation given in Figure \ref{bif2},

\section{Conclusion}

In this study, we investigated the invertibility and monotonicity of stress-strain relations, specifically focusing on the Biot stress tensor-right stretch tensor relation and Rivlin's cube problem. Our primary objective was to determine the conditions under which a unique radial solution exists for Neo-Hooke type materials, where the cube remains a cube under any magnitude of radial stress.

We established that the function 
$h\equiv h_{CG}$ 
defining the Ciarlet-Geymonat energies meets the necessary and sufficient properties for ensuring the existence of a unique radial solution. For the Neo-Hooke-Ciarlet-Geymonat model, we identified both radial and non-radial solutions. In the extension case, non-radial solutions arise, transforming the cube into a parallelepiped, while in compression or below a critical force magnitude 
$\alpha^\flat$, such solutions do not exist. Our analysis revealed that radial solutions maintain local monotonicity up to a critical value 
$\alpha^*$, beyond which bifurcation occurs and monotonicity is lost. This critical value corresponds to the point where invertibility is lost in radial solutions, in terms of principal Biot stresses and principal stretches. For force magnitudes starting from 
$\alpha^*\geq\alpha^\flat$
(below the bifurcation threshold), we identified two classes of non-radial solutions, both appearing discontinuously at 
$\alpha^*$
and then depending continuously on the force intensity. One class of non-radial solutions approaches the bifurcation branch, while the other set of non-radial solutions diverges from it. Numerical tests indicated that the first class of non-radial solutions fails to ensure strong monotonicity, whereas the second class maintains monotonicity, aligning better with physical expectations.

These findings provide insights into the behaviour of stress-strain relations in Neo-Hooke materials and contribute to the understanding of material response under various loading conditions.

Having shown the possible discontinuous nature of multiple solutions with and without symmetry for Rivlin's cube problem, it however remains open whether this can be observed in an experimental setup. It is then natural to inquire as to whether the choice of another elastic energy does not exhibit this surprising response. In other words, this would mean that such insufficiencies stem from the restrictions on the class of elastic energies.

\begin{footnotesize}
	
	\section*{Acknowledgement} The work of Ionel-Dumitrel Ghiba was supported by a grant of the Ministry of Research, Innovation and Digitization, CNCS-UEFISCDI, project number PN-IV-P1-PCE-2023-0915, within PNCDI IV.

\bibliographystyle{plain} 

\addcontentsline{toc}{section}{References}
\begin{thebibliography}{10}
	
	\bibitem{Ball77}
	J.M. Ball.
	\newblock Convexity conditions and existence theorems in nonlinear elasticity.
	\newblock {\em Archive of Rational Mechanics and Analysis}, 63:337--403, 1977.
	
	\bibitem{ball1984differentiability}
	J.M. Ball.
	\newblock Differentiability properties of symmetric and isotropic functions.
	\newblock {\em Duke Mathematical Journal}, 51(1):699--728, 1984.
	
	\bibitem{ball1983bifurcation}
	J.M. Ball and D.G. Schaeffer.
	\newblock Bifurcation and stability of homogeneous equilibrium configurations
	of an elastic body under dead-load tractions.
	\newblock {\em Mathematical Proceedings of the Cambridge Philosophical
		Society}, 94:315--339, 1983.
	
	\bibitem{batra2005treloar}
	R.C. Batra, I.~M\"{u}ller, and P.~Strehlow.
	\newblock Treloar's biaxial tests and {Kearsley's} bifurcation in rubber
	sheets.
	\newblock {\em Mathematics and Mechanics of Solids}, 10(6):705--713, 2005.
	
	\bibitem{borisov2019optimality}
	L.~Borisov, A.~Fischle, and P.~Neff.
	\newblock Optimality of the relaxed polar factors by a characterization of the
	set of real square roots of real symmetric matrices.
	\newblock {\em Zeitschrift f\"ur Angewandte Mathematik und Mechanik},
	99(6):\,e201800120, 2019.
	
	\bibitem{chen1987stability}
	Y.C. Chen.
	\newblock Stability of homogeneous deformations of an incompressible elastic
	body under dead-load surface tractions.
	\newblock {\em Journal of Elasticity}, 17(3):223--248, 1987.
	
	\bibitem{chen1995stability}
	Y.C. Chen.
	\newblock Stability of homogeneous deformations in nonlinear elasticity.
	\newblock {\em Journal of Elasticity}, 40(1):75--94, 1995.
	
	\bibitem{chen1996stability}
	Y.C. Chen.
	\newblock Stability and bifurcation of homogeneous deformations of a
	compressible elastic body under pressure load.
	\newblock {\em Mathematics and Mechanics of Solids}, 1(1):57--72, 1996.
	
	\bibitem{ciarlet1982quelques}
	Ph.G. Ciarlet.
	\newblock {\em Quelques remarques sur les probl{\`e}mes d'existence en
		{\'e}lasticit{\'e} non lin{\'e}aire}.
	\newblock PhD thesis, INRIA, 1982.
	
	\bibitem{ciarlet1982lois}
	Ph.G. Ciarlet and G.~Geymonat.
	\newblock Sur les lois de comportement en {\'e}lasticit{\'e} non lin{\'e}aire
	compressible.
	\newblock {\em Comptes Rendus de l'Acad\'{e}mie des Sciences - Series I -
		Mathematics}, 295:423--426, 1982.
	
	\bibitem{fischle2017grioli}
	A.~Fischle and P.~Neff.
	\newblock {Grioli's} theorem with weights and the relaxed-polar mechanism of
	optimal cosserat rotations.
	\newblock {\em Rendiconti Lincei-Matematica E Applicazioni}, 28(3):573--601,
	2017.
	
	\bibitem{galewski2014conditions}
	M.~Galewski, E.Z. Galewska, and E.~Schmeidel.
	\newblock Conditions for having a diffeomorphism between two {Banach} spaces.
	\newblock {\em Electronic Journal of Differential Equations}, 2014(99):1--6,
	2014.
	
	\bibitem{gent2005elastic}
	A.N. Gent.
	\newblock Elastic instabilities in rubber.
	\newblock {\em International Journal of Non-Linear Mechanics}, 40(2):165--175,
	2005.
	
	\bibitem{hill1968constitutivea}
	R.~Hill.
	\newblock On constitutive inequalities for simple materials\,-\, {I}.
	\newblock {\em Journal of the Mechanics and Physics of Solids}, 16(4):229--242,
	1968.
	
	\bibitem{hill1968constitutiveb}
	R~Hill.
	\newblock On constitutive inequalities for simple materials {II}.
	\newblock {\em Journal of the Mechanics and Physics of Solids}, 16(5):315--322,
	1968.
	
	\bibitem{hill1970constitutive}
	R.~Hill.
	\newblock Constitutive inequalities for isotropic elastic solids under finite
	strain.
	\newblock {\em Proceedings of the Royal Society of London. Series A,
		Mathematical and Physical Sciences}, 314(1519):457--472, 1970.
	
	\bibitem{katriel2016mountain}
	G.~Katriel.
	\newblock Mountain pass theorems and global homeomorphism theorems.
	\newblock {\em Annales de l'Institut Henri Poincare (C) Non Linear Analysis},
	11:189--209, 1994.
	
	\bibitem{knowles1976failure}
	J.K. Knowles\ and E.~Sternberg.
	\newblock On the failure of ellipticity of the equations for finite
	elastostatic plane strain.
	\newblock {\em Archive of Rational Mechanics and Analysis}, 63(4):321--336,
	1976.
	
	\bibitem{knowles1978failure}
	J.K. Knowles and E.~Sternberg.
	\newblock On the failure of ellipticity and the emergence of discontinuous
	deformation gradients in plane finite elastostatics.
	\newblock {\em Journal of Elasticity}, 8(4):329--379, 1978.
	
	\bibitem{Marsden83}
	J.E. Marsden and J.R. Hughes.
	\newblock {\em Mathematical {F}oundations of {E}lasticity.}
	\newblock Prentice-Hall, Englewood Cliffs, New Jersey, 1983.
	
	\bibitem{NeffMartin14}
	R.~Martin and P.~Neff.
	\newblock Some remarks on monotonicity of primary matrix functions on the set
	of symmetric matrices.
	\newblock {\em Archive of Applied Mechanics}, 85:1761--1778, 2015.
	
	\bibitem{MartinVossGhibaNeff}
	R.J. Martin, J.~Voss, I.D. Ghiba, M.V. {d'Agostino}, and P.~Neff.
	\newblock Monotonicity of isotropic tensor functions on the set of symmetric
	matrices: {Hill's} generalization of the {Chandler-Davis-Lewis} convexity
	theorem revised.
	\newblock {\em in preparation}.
	
	\bibitem{mihai2019likely}
	L.A. Mihai, T.E. Woolley, and A.~Goriely.
	\newblock Likely equilibria of the stochastic {Rivlin} cube.
	\newblock {\em Philosophical Transactions of the Royal Society A},
	377(2144):20180068, 2019.
	
	\bibitem{Neff_Osterbrink_Martin_hencky13}
	P.~Neff, B.~Eidel, and R.~J. Martin.
	\newblock Geometry of logarithmic strain measures in solid mechanics.
	\newblock {\em Archive of Rational Mechanics and Analysis}, 222:507--572,
	2016.
	
	\bibitem{NeffGhibaLankeit}
	P.~Neff, I.~D. Ghiba, and J.~Lankeit.
	\newblock The exponentiated {H}encky-logarithmic strain energy. {P}art {I}:
	{C}onstitutive issues and rank--one convexity.
	\newblock {\em Journal of Elasticity}, 121:143--234, 2015.
	
	\bibitem{neff2013grioli}
	P.~Neff, J.~Lankeit, and A.~Madeo.
	\newblock On {G}rioli's minimum property and its relation to {C}auchy's polar
	decomposition.
	\newblock {\em International Journal of Engineering Science}, 80:207--217,
	2014.
	
	\bibitem{Ogden83}
	R.W. Ogden.
	\newblock {\em Non-{L}inear {E}lastic {D}eformations.}
	\newblock Mathematics and its Applications. Ellis Horwood, Chichester, 1983.
	
	\bibitem{reese1997material}
	S.~Reese and P.~Wriggers.
	\newblock Material instabilities of an incompressible elastic cube under
	triaxial tension.
	\newblock {\em International Journal of Solids and Structures},
	34(26):3433--3454, 1997.
	
	\bibitem{rivlin1974stability}
	R.S. Rilvin.
	\newblock Stability of pure homogeneous deformations of an elastic cube under
	dead loading.
	\newblock {\em Quarterly of Applied Mathematics}, 32(3):265--271, 1974.
	
	\bibitem{rivlin2003dead}
	R.~Rivlin and M.F. Beatty.
	\newblock Dead loading of a unit cube of compressible isotropic elastic
	material.
	\newblock {\em Zeitschrift f\"ur Angewandte Mathematik und Physik},
	54(6):954--963, 2003.
	
	\bibitem{soldatos2006stability}
	K.P. Soldatos.
	\newblock On the stability of a compressible {Rivlin's} cube made of
	transversely isotropic material.
	\newblock {\em IMA Journal of Applied Mathematics}, 71(3):332--353, 2006.
	
	\bibitem{tarantino2008homogeneous}
	A.~Tarantino.
	\newblock Homogeneous equilibrium configurations of a hyperelastic compressible
	cube under equitriaxial dead-load tractions.
	\newblock {\em Journal of Elasticity}, 92(3):227--254, 2008.
	
	\bibitem{Truesdell65}
	C.~Truesdell and W.~Noll.
	\newblock The non-linear field theories of mechanics.
	\newblock In S.~Fl\"ugge, editor, {\em Handbuch der {P}hysik}, volume III/3.
	Springer, Heidelberg, 1965.
	
	\bibitem{wan1983symmetry}
	Y.H. Wan and J.E. Marsden.
	\newblock Symmetry and bifurcation in three-dimensional elasticity.
	\newblock {\em Archive of Rational Mechanics and Analysis}, 84(3):203--233,
	1983.
	
\end{thebibliography}

\appendix
\section{General notation}\label{notations}
\textbf{Inner product} \\
For $a,b\in\R^n$ we let $\langle {a},{b}\rangle_{\R^n}$  denote the scalar product on $\R^n$ with associated vector norm $\norm{a}_{\R^n}^2=\langle {a},{a}\rangle_{\R^n}$. We denote by $\R^{n\times n}$ the set of real $n\times n$ second order tensors, written with capital letters. The standard Euclidean scalar product on $\R^{n\times n}$ is given by
$\langle {X},{Y}\rangle_{\R^{n\times n}}=\tr{(X Y^T)}$, where the superscript $^T$ is used to denote transposition. Thus the Frobenius tensor norm is $\norm{X}^2=\langle {X},{X}\rangle_{\R^{n\times n}}$, where we usually omit the subscript $\R^{n\times n}$ in writing the Frobenius tensor norm. The identity tensor on $\R^{n\times n}$ will be denoted by $\id$, so that $\tr{(X)}=\langle {X},{\id}\rangle$. \\
\\
\textbf{Frequently used spaces} 
	\begin{itemize}
	\item $\Sym(n), \rm \Sym^+(n)$ and $\Sym^{++}(n)$ denote the symmetric, positive semi-definite symmetric and positive definite symmetric tensors respectively.
	\item ${\rm GL}(n):=\{X\in\R^{n\times n}\;|\det{X}\neq 0\}$ denotes the general linear group.
	\item ${\rm GL}^+(n):=\{X\in\R^{n\times n}\;|\det{X}>0\}$ is the group of invertible matrices with positive determinant.
	\item ${\rm SL}(n):=\{X\in {\rm GL}(n)\;|\det{X}=1\}$,
	\item $\mathrm{O}(n):=\{X\in {\rm GL}(n)\;|\;X^TX=\id\}$,
	\item ${\rm SO}(n):=\{X\in {\rm GL}(n,\R)\;|\; X^T X=\id,\;\det{X}=1\}$,
	\item $\mathfrak{so}(3):=\{X\in\mathbb{R}^{3\times3}\;|\;X^T=-X\}$ is the Lie-algebra of skew symmetric tensors.
	\item $\mathfrak{sl}(3):=\{X\in\mathbb{R}^{3\times3}\;|\; \tr({X})=0\}$ is the Lie-algebra of traceless tensors.
	\item The set of positive real numbers is denoted by $\R_+:=(0,\infty)$, while $\overline{\R}_+=\R_+\cup \{\infty\}$.
	\end{itemize}
\textbf{Frequently used tensors}
	\begin{itemize}
	\item $C=F^T \, F$ is the right Cauchy-Green strain tensor.
	\item $B=F\, F^T$ is the left Cauchy-Green (or Finger) strain tensor.
	\item $U = \sqrt{F^T \, F} \in \Sym^{++}(3)$ is the right stretch tensor, i.e.~the unique element of ${\rm Sym}^{++}(3)$ with $U^2=C$.
	\item $V = \sqrt{F \, F^T} \in \Sym^{++}(3)$ is the left stretch tensor, i.e.~the unique element of ${\rm Sym}^{++}(3)$ with $V^2=B$.
	\item We also have the polar decomposition $F = R \, U = V R \in {\rm GL}^+(3)$ with an orthogonal matrix $R \in \OO(3)$.
	\end{itemize}
\textbf{Further definitions and conventions}
	\begin{itemize}
	\item For $X\in {\rm GL}(3)$, $\Cof X = (\det X)X^{-T}$ is the cofactor of $X\in {\rm GL}(3)$, while  ${\rm Adj}({X})$ denotes the tensor of transposed cofactors.  
	\item For vectors $\xi,\eta\in\mathbb{R}^3$, we have the tensor product $(\xi\otimes\eta)_{ij}=\xi_i\,\eta_j$. 
	\item For vectors $v=\left(v_1,v_2,v_3\right)^T\ \in\R^3, $ we define
	$
	{\rm diag}\, v=\left(
	\begin{array}{ccc}
	v_1 & 0 & 0 \\
	0 & v_2 & 0 \\
	0 & 0 & v_3 \\
	\end{array}
	\right).
	$
	\item The Fr\'echet derivative of a function $W:\mathbb{R}^{3\times 3}\to \mathbb{R}$ at $F\in \mathbb{R}^{3\times 3}$ applied to the tensor-valued increment $H$ is denoted by ${\rm D}_F[W(F)]. H$. Similarly, ${\rm D}_F^2[W(F)]. (H_1,H_2)$ is the bilinear form induced by the second Fr\'echet derivative of the function $W$ at $F$ applied to $(H_1,H_2)$.
	\item Let $\Omega\subset{\R^3}$  be a  bounded open domain with Lipschitz boundary $\partial\Omega$.   The usual Lebesgue spaces of square-integrable functions, vector or tensor fields on $\Omega$ with values in $\mathbb{R}$, $\mathbb{R}^3$, $\mathbb{R}^{3\times 3}$ or ${\rm SO}(3)$, respectively will be denoted by ${\rm L}^2(\Omega;\mathbb{R})$, ${\rm L}^2(\Omega;\mathbb{R}^3)$, ${\rm L}^2(\Omega; \mathbb{R}^{3\times 3})$ and ${\rm L}^2(\Omega; {\rm SO}(3))$, respectively. 
	\item For vector fields $u=\left(    u_1, u_2,u_3\right)$ with  $u_i\in {\rm H}^{1}(\Omega)$, $i=1,2,3$,
	we define
	$
	{\rm D} \,u:=\left(
	{\rm D}  u_1\,|\,
	{\rm D} u_2\,|\,
	{\rm D} u_3
	\right)^T.
	$
	\end{itemize}
	 
\end{footnotesize}

\end{document}